\documentclass[a4paper,11pt]{article}
\usepackage{amsfonts}
\usepackage{amstext}
\usepackage{amsthm}
\usepackage{amsmath}
\usepackage{amssymb}
\usepackage{latexsym}
\usepackage{graphicx}
\usepackage{caption}
\usepackage[hidelinks]{hyperref}
\usepackage{enumitem}
\usepackage{mathtools}
\usepackage{rotating}
\usepackage{mathrsfs}

\usepackage[utf8]{inputenc}
\newcommand{\R}{\Bbb{R}}
\newcommand{\N}{\Bbb{N}}

\newcommand{\s}{\Bbb{S}}

\newtheorem{teor}{Theorem}[section]
\newtheorem{propo}{Proposition}[section]

\newtheorem{lema}{Lemma}[section]
\newtheorem{cor}{Corollary}[section]

\newtheorem{ex}{Example}[section]

\newcommand{\n}{\noindent}

\begin{document}

\title{%Positively homogeneous anisotropies in low dimensions and related anisotropic spectral optimization\\

Sharp estimates for fundamental frequencies of elliptic operators generated by asymmetric seminorms in low dimensions
}

\author{\textbf{Julian Haddad \footnote{\textit{E-mail addresses}:
jhaddad@us.es (J. Haddad)}}\\ {\small\it Departamento de Análisis Matemático,
Universidad de Sevilla,}\\ {\small\it Apartado de Correos 1160, Sevilla 41080, Spain}\\
\textbf{Raul Fernandes Horta \footnote{\textit{E-mail addresses}: raul.fernandes.horta@gmail.com (R. F. Horta)}}\\ {\small\it Departamento de Matem\'{a}tica,
Universidade Federal de Minas Gerais,}\\ {\small\it Caixa Postal 702, 30123-970, Belo Horizonte, MG, Brazil}\\
\textbf{Marcos Montenegro \footnote{\textit{E-mail addresses}: montene@mat.ufmg.br (M. Montenegro)}}\\ {\small\it Departamento de Matem\'{a}tica,
Universidade Federal de Minas Gerais,}\\ {\small\it Caixa Postal 702, 30123-970, Belo Horizonte, MG, Brazil}
}

\date{}{%{\it Preprint - December 12, 2009}}

\maketitle

\markboth{abstract}{abstract}

\hrule \vspace{0,2cm}

\n {\bf Abstract}

We establish new sharp asymmetric Poincaré inequalities in one-dimension with the computation of optimal constants and characterization of extremizers. Using the one-dimensional theory we develop a comprehensive study on fundamental frequencies in the plane and related spectral optimization in the very general setting of positively homogeneous anisotropies.

\vspace{0.5cm}

\hrule

\section{Overview and main statements}\label{sec1}

A line of research that has gained a lot of attention within spectral optimization is the study of sharp uniform estimates of fundamental frequencies associated to a family of elliptic operators defined in a given bounded domain. In precise terms, let $\Omega \subset \R^n$ be a bounded domain and consider a set ${\cal E}(\Omega)$ of second-order elliptic operators on $\Omega$ so that each operator ${\cal L} \in {\cal E}(\Omega)$ admits at least a non-zero fundamental frequency $\lambda_1({\cal L})$ with respect some a priori fixed boundary condition (e.g. Dirichlet, Neumann, Robin).

Assume that there exists a way of ``measuring" each element ${\cal L} \in {\cal E}(\Omega)$ by some quantity denoted say by ${\cal M}({\cal L})$. Two cornerstone questions in spectral optimization related to ${\cal E}(\Omega)$ are:
\begin{enumerate}[label=(\Roman*)]
\item\label{itemA} Are there explicit optimal constants $\Lambda^{\min}(\Omega)$ and $\Lambda^{\max}(\Omega)$ such that
\begin{equation} \label{SSE}
\Lambda^{\min}(\Omega)\, {\cal M}({\cal L}) \leq \lambda_1({\cal L}) \leq \Lambda^{\max}(\Omega)\, {\cal M}({\cal L})
\end{equation}
for every operator ${\cal L} \in {\cal E}(\Omega)$?

\item\label{itemB} Does any inequality in \eqref{SSE} become equality for some operator ${\cal L} \in {\cal E}(\Omega)$?
\end{enumerate}
The left sharp inequality in \eqref{SSE} can be seen as an analytical counterpart of the classical Faber-Krahn geometric inequality for the fundamental frequency $\lambda_1(\Omega)$ of the Laplace operator $- \Delta$ (cf. \cite{F, Kr}):
\begin{equation} \label{FK1}
\vert B \vert^{2/n} \lambda_1(B) \, \vert \Omega \vert^{-2/n} \leq \lambda_1(\Omega),
\end{equation}
where $B$ denotes the unit $n$-Euclidean ball and $\vert \cdot \vert$ stands for the Lebesgue measure of a measurable subset of $\R^n$. In contrast to \eqref{SSE}, the operator $- \Delta$ is fixed in \eqref{FK1}, whereas the bounded domain $\Omega$ varies freely.

Other important inequalities motivated by physical applications and closely related to \eqref{FK1} regarding the first non-zero eigenvalue of the Laplace operator under different boundary conditions are: Bossel-Daners inequality \cite{B, BD, D, DK, DPG, DPP} (Robin condition), Szegö-Weinberger inequality \cite{BH, S, W} (Neumann condition) and Brock-Weinstock inequality \cite{Br, We} (Steklov condition). Faber-Krahn type inequalities have also been proved for other elliptic equations/operators as for example the Schrödinger equation \cite{DCH}, the $p$-Laplace operator \cite{AFT}, the affine $p$-Laplace operator \cite{HJM5}, anisotropic operators \cite{BFK}, the fractional Laplace operator \cite{BLP} and local/non-local type operators \cite{BDVV}.

There is a vast literature within spectral optimization addressed to questions closely related to \ref{itemA} and \ref{itemB} involving second-order linear elliptic operators. Most of the contributions can be found in the following long, but far from complete, list of references: regarding the Schrödinger operator we refer to \cite{CMZ, EK, E, K, KS, N, T, ZW} for the one-dimensional case and \cite{AM, BBV, BGRV, CGIKO, Eg, H, MRR} and Chapter 9 of \cite{H2}, for higher dimensions and for the conductivity operator we refer to \cite{Ba, CL, EKo, Tro}.

On the other hand, fairly little is known about spectral optimization with respect to nonlinear elliptic operators, the subject began to be addressed very recently in the works \cite{HM1} and \cite{HM2} where, in the first one, \ref{itemA} and \ref{itemB} are completely solved in the plane on the class ${\cal E}(\Omega)$ of elliptic operators of quadratic type and, in the second one on the class ${\cal E}(\Omega)$ of anisotropic elliptic operators generated by seminorms as well in the plane.

The present work investigates \ref{itemA} and \ref{itemB} for the quite general class ${\cal E}(\Omega)$ of anisotropic elliptic operators generated by asymmetric seminorms for any fixed bounded domain $\Omega$ in the line and in the plane.

In order to introduce the set ${\cal E}(\Omega)$ for a given bounded domain $\Omega \subset \R^n$ with $n = 1$ or $n = 2$, consider the class of functions, called {\it asymmetric anisotropies}:
\[
\mathscr{H} = \left\{H \colon \R^n \to \R \colon H \ \text{is non-negative, convex and positively $1$-homogeneous}\right\},
\]
where the space $\R^n$ is assumed to be endowed with the usual Euclidean norm denoted by $|\cdot|$.

Recall that a function $H$ is said to be convex if it satisfies
\[
H(tx + (1-t) y) \leq tH(x) + (1-t)H(y)
\]
for every $0 < t < 1$ and $x,y \in \R^n$, and is positively $1$-homogeneous if
\[
H(tx) = t H(x)
\]
for every $t > 0$ and $x \in \R^n$.

The elements of $\mathscr{H}$ are also called asymmetric seminorms or, if they are positive in $\R^n \setminus \{0\}$, asymmetric norms. For $n = 1$, we clearly have
\[
\mathscr{H} = \left\{H_{a,b} \colon\, H_{a,b}(t):= a t^+ + b t^-,\, a,b \geq 0\right\},
\]
where $t^\pm := \max\{\pm t,0\}$.

For each number $p > 1$ and anisotropy $H \in \mathscr{H}$, define the $H$-anisotropic $p$-Laplace operator as
\[
-\Delta_p^H u := -{\rm div}\left(H^{p-1}\left(\nabla u\right)\nabla H\left(\nabla u\right)\right).
\]
Corresponding to $H$, we have the anisotropic $L^p$ energy ${\cal E}_{p}^{H}: W^{1,p}_{0}(\Omega) \to \R$ defined by
\[
{\cal E}_{p}^{H}(u) := \int_{\Omega} H^{p}(\nabla u) \, dx,
\]
where $\nabla u$ stands for the weak gradient of $u$ (which as usual often will be denoted by $u^\prime$ when $n = 1$) and $W^{1,p}_{0}(\Omega)$ denotes the completion of compactly supported smooth functions in $\Omega$ with respect to the norm
\[
\lVert u \rVert_{W^{1,p}_{0}(\Omega)} = \left(\int_{\Omega} | \nabla u|^{p} \, dx\right)^{\frac{1}{p}}.
\]
The least energy level associated to ${\cal E}_{p}^{H}$ on the unit sphere in $L^p(\Omega)$ is defined as
\begin{equation}\label{frequency}
\lambda_{p}^{H}(\Omega) := \inf\left\{{\cal E}_{p}^{H}(u) \colon\ u \in W^{1,p}_0(\Omega), \ \lVert u \rVert_{p} = 1\right\},
\end{equation}
where
\[
\lVert u \rVert_p = \left(\int_{\Omega} |u|^{p} \, dx\right)^{\frac{1}{p}}.
\]
Knowing whether $\lambda_{p}^{H}(\Omega)$ is a fundamental frequency associated to the operator $-\Delta_p^H$ on $W_0^{1,p}(\Omega)$ is a problem to be discussed in the paper and, as we shall see, the answer depends on the regularity of $H$. By a fundamental frequency we mean that there exists a non-zero function $u_p \in W^{1,p}_0(\Omega)$ such that $\nabla H\left(\nabla u_p\right)$ is well-defined almost everywhere, is Lebesgue measurable in $\Omega$ and such that
\[
\int_{\Omega} H^{p-1}\left(\nabla u_p\right)\nabla H\left(\nabla u_p\right)\cdot\nabla\varphi\, dx = \lambda_{p}^{H}(\Omega)\int_{\Omega}|u_p|^{p-2}u_p\varphi\, dx
\]
for every test function $\varphi \in W^{1,p}_0(\Omega)$. The finiteness of the above left-hand side is immediate since asymmetric anisotropies $H$ are Lipschitz.

The set of interest ${\cal E}(\Omega)$ is given by
\[
{\cal E}(\Omega):= \{-\Delta_p^H:\, H \in \mathscr{H} \setminus \{0\}\ \text{such that}\ \lambda_{p}^{H}(\Omega)\ \text{is a fundamental frequency}\}.
\]
Observe that $-\Delta_p^H \in {\cal E}(\Omega)$ for every asymmetric norm $H$ of class $C^1$ on $\R^n \setminus \{0\}$. In particular, for $n = 1$, ${\cal E}(\Omega)$ contains the one-dimensional asymmetric $p$-Laplace operator $-\Delta_p^{H_{a,b}}$, where
\[
\Delta_p^{H_{a,b}}\, u = \left(a^{p}[(u')^{+}]^{p-1}-b^{p}[(u')^{-}]^{p-1}\right)'
\]
for constants $a, b > 0$, and for $n = 2$, ${\cal E}(\Omega)$ contains the 2-dimensional $p$-Laplace operator $- \Delta_p$ where $\Delta_p =: \Delta_p^H$ with $H(x,y) = \vert (x,y) \vert$, or explicitly
\[
\Delta_p u = {\rm div}(\vert \nabla u \vert^{p-2} \nabla u).
\]

Our first results concern the structure of the set ${\cal E}(\Omega)$. For instance, when $n = 2$ and $H$ is an asymmetric norm we prove that $-\Delta_p^H \in {\cal E}(\Omega)$ for every bounded domain $\Omega$ if and only if $H$ is of $C^1$ class in $\R^2 \setminus \{(0,0)\}$. For asymmetric seminorms we also show that the operator $-\Delta_p^H$ does not belong to ${\cal E}(\Omega)$ for any, $\Omega$ in the cases $H(t) = H_{a,0}(t)$ for $n = 1$ and $H(x,y) = H_{a,0}(y)$ for $n = 2$ when $a > 0$. Also, in this context, we provide a characterization of all domains $\Omega \subset \R^2$ such that $-\Delta_p^H \in {\cal E}(\Omega)$ for $H(x,y) = H_{a,b}(y)$ with $a, b > 0$.

We begin by presenting the complete picture in dimension $n = 1$. Assume without loss of generality that $\Omega = (-T, T)$ for some $T > 0$. For each $a, b \geq 0$, set $\lambda_{p}^{a,b}(\Omega) := \lambda_{p}^{H_{a,b}}(\Omega)$. Denote also by $\lambda_{1,p}(\Omega)$ the fundamental frequency of $- \Delta_p$ in $\Omega$ and let $\varphi_{p} \in W^{1,p}_0(\Omega)$ be the corresponding principal eigenfunction normalized by $\Vert \varphi_{p} \Vert_p = 1$.
Of course, we rule out the trivial case $a=b=0$, and by symmetry in $a$ and $b$ we may assume that $a>0$.
\begin{teor} \label{T1.1} Let $p > 1$ and let $T > 0$. For any $a > 0$ and $b \geq 0$, we have:
\begin{itemize}
\item[(i)] $\lambda_{p}^{a,b}(-T,T) = \displaystyle \left(\frac{a+b}{2}\right)^{p}\lambda_{1,p}(-T,T)$;
\item[(ii)] $u$ is an extremizer of \eqref{API} if and only if either $u = u_p$ or $u = v_p$, where (see Figure \ref{fig_functions_u})
\begin{equation*}
u_p(t) =
\begin{cases}
\begin{aligned}
     & c \varphi_{p}\left(\frac{T(t-t_{0})}{T+t_{0}}\right)\ \ \text{if} \ t \in (-T,t_{0}), \\
     & c \varphi_{p}\left(\frac{T(t-t_{0})}{T-t_{0}}\right)\ \ \text{if} \ t \in [t_{0},T)
\end{aligned}
\end{cases}
\end{equation*}
and
\begin{equation*}
v_p(t) =
\begin{cases}
\begin{aligned}
     & -c \varphi_{p}\left(\frac{T(t+t_{0})}{T-t_{0}}\right)\ \ \text{if} \ t \in (-T,-t_{0}), \\
     & -c \varphi_{p}\left(\frac{T(t+t_{0})}{T+t_{0}}\right)\ \ \text{if} \ t \in [-t_{0},T),
\end{aligned}
\end{cases}
\end{equation*}
where $c$ is any positive constant and
\[
t_{0} = \left(\frac{a-b}{a+b}\right)T.
\]
\end{itemize}
Consequently, since $u \in W_0^{1,p}(-T,T)$ only when $b > 0$, $\lambda_{p}^{a,b}(-T,T)$ is a fundamental frequency if and only if $b > 0$. In other words,
\[
{\cal E}(\Omega) = \{-\Delta_p^{H_{a,b}}:\, a, b > 0\}.
\]
\end{teor}
\begin{figure}
	\includegraphics[width=\textwidth]{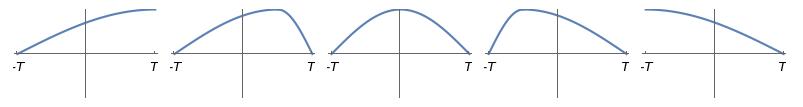}
	\caption{The functions $u_p$ for $p=2$ and the values $b=1-a$ and $a = 0, \frac 14, \frac 12, \frac 34, 1$.}
	\label{fig_functions_u}
\end{figure}

In particular, we have:
\begin{cor} \label{C1}
Let $p > 1$, let $T > 0$ and let $H \in \mathscr{H}$. The level $\lambda_{p}^{H}(-T,T)$ is positive if and only if $H \neq 0$.
\end{cor}

For related results in dimension $n = 2$, it is convenient to decompose $\mathscr{H}$ into the subsets of positive and degenerate anisotropies to be denoted respectively by
\begin{align*}
	\mathscr{H}_P &:= \left\{H \in \mathscr{H}:\ H(x,y) > 0\ \forall (x,y) \neq (0,0)\right\},\\
	\mathscr{H}_D &:= \left\{H \in \mathscr{H}:\ H(x,y) = 0\ \text{for some}\ (x,y) \neq (0,0)\right\}.
\end{align*}

\begin{teor} \label{T2.1}
Let $p > 1$ and let $H \in \mathscr{H}_P$. The assertions are equivalent:
\begin{itemize}
\item[(a)] $\lambda_{p}^{H}(\Omega)$ is a fundamental frequency for any bounded domain $\Omega \subset \R^2$;
\item[(b)] $H$ is of $C^1$ class in $\R^2 \setminus \{(0,0)\}$.
\end{itemize}
\end{teor}
We now concentrate on the degenerate case and provide complete answers for two subsets of $\mathscr{H}_D$.

Let $H \in \mathscr{H}_D \setminus \{0\}$ be such that $\ker(H):= \{(x,y) \in \R^2:\, H(x,y) = 0\}$ is a line or a half-plane. In each case there is a rotation matrix $A$ such that
\begin{equation}\label{Amatrix}
H_{A}(x,y):= (H \circ A)(x,y) = H_{a,b}(y) = a y^+ + b y^-,
\end{equation}
where $a, b > 0$ when $\ker(H)$ is a line, and $a > 0$ and $b = 0$ when $\ker(H)$ is a half-plane (see Proposition \ref{P5}).

The next results describe the two situations separately.
\begin{teor}[The case $b=0$] \label{T3.1}
Let $p > 1$ and let $\Omega$ be any bounded domain in $\R^2$. If $\ker(H)$ is a half-plane then the level $\lambda^H_{p}(\Omega)$ is never a fundamental frequency.
\end{teor}

\begin{teor}[The case $b>0$] \label{T4.1}
Let $p > 1$, $\Omega$ be any $C^{0,1}$ bounded domain in $\R^2$, assume $\ker(H)$ is a line and let $A$ be as in \eqref{Amatrix}. Denote $\Omega_A:= A^{T}(\Omega)$. The assertions are equivalent:
\begin{itemize}
\item[(a)] $\lambda^H_{p}(\Omega)$ is a fundamental frequency;

\item[(b)] There are bounded open intervals $I^\prime \subset I \subset \mathbb{R}$, a $C^{0,1}$ sub-domain $\Omega^\prime \subset \Omega$ and a number $L > 0$ such that:
\begin{itemize}

\item[(i)] $\Omega_A \subset I \times \mathbb{R}$,

\item[(ii)] Each connected component of $\{ y \in \mathbb{R}:\, (x,y) \in \Omega_A) \}$ has length at most $L$ for every $x \in I$,

\item[(iii)] The set $\{ y \in \mathbb{R}:\, (x,y) \in \Omega_A^\prime \}$ is an interval of length $L$ for every $x \in I^\prime$.
\end{itemize}
\end{itemize}
In that case, $L$ is the number so that $\lambda^H_{p}(\Omega) = \lambda^{a,b}_{p}(0,L)$.
\end{teor}

Let us now turn our attention to the problems \ref{itemA} and \ref{itemB}. By the definition of ${\cal E}(\Omega)$ we can denote
\[
\lambda_1(-\Delta_p^H) := \lambda^H_{p}(\Omega).
\]
Among different possible ways of ``measuring" operators in ${\cal E}(\Omega)$, we choose to work with
\[
{\cal M}(-\Delta_p^H) = \Vert H \Vert^p,
\]
where
\[
\Vert H \Vert:= \max_{\vert x \vert = 1} H(x).
\]
For this choice, the inequalities in \eqref{SSE} take the form
\begin{equation} \label{SSE1}
\Lambda^{\min}(\Omega)\, \Vert H \Vert^p \leq \lambda_1(-\Delta_p^H) \leq \Lambda^{\max}(\Omega)\, \Vert H \Vert^p.
\end{equation}
Furthermore, as can be easily checked, ${\cal M}$ is positively homogeneous, that is
\[
	{\cal M}(- \lambda \Delta_p^H) = {\cal M}(- \Delta_p^{\lambda^{1/p}H}) = \lambda\, {\cal M}(- \Delta_p^H)
\]
for every $\lambda > 0$. Consequently, the optimal constants in \eqref{SSE1} are given by
\begin{align*}
	\Lambda^{\max}(\Omega) &= \sup\left\{\lambda_1(-\Delta_p^H):\, -\Delta_p^H \in {\cal E}(\Omega)\ \ {\rm and}\ \ \Vert H \Vert = 1 \right\}, \\
	\Lambda^{\min}(\Omega) &= \inf\left\{\lambda_1(-\Delta_p^H):\, -\Delta_p^H \in {\cal E}(\Omega)\ \ {\rm and}\ \ \Vert H \Vert = 1 \right\}.
\end{align*}

In dimension $n = 1$, Theorem \ref{T1.1} provides the complete solution for the questions \ref{itemA} and \ref{itemB}. Since $\Vert H_{a,b} \Vert= \max\{a,b\}$, we have:
\begin{cor} \label{C2}
For any $p > 1$ and $T > 0$, we have
\begin{equation} \label{SSE2}
\Lambda^{\min}(-T,T)\, \max\{a,b\}^p \leq \lambda_1(-\Delta_p^{H_{a,b}}) \leq \Lambda^{\max}(-T,T)\, \max\{a,b\}^p
\end{equation}
for every $-\Delta_p^{H_{a,b}} \in {\cal E}(\Omega)$, where
\begin{align*}
	\lambda_{1,p}^{\max}(-T,T) &= \lambda_{1,p}(-T,T),\\
	\lambda_{1,p}^{\min}(-T,T) &= \displaystyle \frac{1}{2^p} \lambda_{1,p}(-T,T).
\end{align*}
Moreover, all extremal operators for the right inequality in \eqref{SSE2} are given by the non-zero multiples of $-\Delta_p$, and for the left one they are given by $- \Delta_p^{H_{a,0}}$ and $- \Delta_p^{H_{0,b}}$ for all $a, b > 0$, which do not belong to ${\cal E}(\Omega)$.
\end{cor}

We now concentrate on related results in dimension $n = 2$. The next theorem states that the right sharp inequality in \eqref{SSE1} is rigid under some regularity assumptions on the boundary of $\Omega$, and gives explicitly the constant $\Lambda^{\max}(\Omega)$.
\begin{teor} \label{T5.1}
Let $p > 1$ and let $\Omega$ be any bounded domain in $\R^2$. The upper optimal constant $\Lambda^{\max}(\Omega)$ is given by $\lambda_{1,p}(\Omega)$. Moreover, all non-zero multiples of $-\Delta_p$ are extremal operators for the second inequality in \eqref{SSE1} and only them provided that $\partial \Omega$ satisfies the Wiener condition (see the definition for instance in \cite{KM, LM, Ma}).
\end{teor}

Before focusing on the first sharp inequality in \eqref{SSE1}, we recall the notion of optimal anisotropic design for bounded domains, introduced in \cite{HM2}. Consider the width function $L: [0,\pi] \rightarrow \R$ defined by $L(\theta) = L_\theta$, where
\[
	L_\theta:= \sup_{(x,y) \in \R^{2}} \sup\left\{|J| \colon J \subset \left\{t(\cos{\theta},\sin{\theta}) + (x,y) \colon t \in \R\right\}\cap\Omega, \ J \ \text{is connected}\right\}.
\]
A bounded domain $\Omega$ is said to have an optimal anisotropic design if $L: [0,\pi] \rightarrow \R$ has a global maximum.

The concept of optimal design is motivated by the following statement:
\begin{teor} \label{T6.1}
Let $p > 1$ and let $\Omega$ be any bounded domain in $\R^2$. The lower optimal constant $\Lambda^{\min}(\Omega)$ is always positive and given by
\[
\Lambda^{\min}(\Omega) = \inf_{\theta \in [0,2 \pi]} \lambda_{1,p}(0, 2L_\theta).
\]
Moreover, the infimum is attained if and only if $\Omega$ has optimal anisotropic design. In this case, if $\theta_0$ is a global maximum point of the function $L$ then all non-zero multiples of $-\Delta_p^{H_0}$, with
\[
H_0(x,y) = (x\cos{\theta_0} + y\sin{\theta_0})^+,
\]
are extremal operators for the first inequality in \eqref{SSE1}, which do not belong to ${\cal E}(\Omega)$.
\end{teor}

We conclude the introduction with the outline of paper organized by section.
\begin{description}
\item[Section~\ref{sec2}:] We establish a family of new asymmetric Poincaré inequalities in dimension 1 with computation of optimal constants and characterization of all corresponding extremizers.
\item[Section~\ref{sec3}:] It is dedicated to the characterization of fundamental frequencies in dimension $2$ for the positive asymmetric anisotropies.
\item[Section~\ref{sec4}:] It is devoted to the characterization of fundamental frequencies in dimension $2$ for some degenerate anisotropies.
\item[Section~\ref{sec5}:] As an application of Section 3, we present an extended version of the well-known isoperimetric maximization conjecture to fundamental frequencies.
\item[Section~\ref{sec6}:] It is entirely addressed to the study of the questions \ref{itemA} and \ref{itemB} in dimension $2$.
\end{description}

\section{Asymmetric Poincaré inequalities for $n = 1$}\label{sec2}

In this section we carry out the complete study of sharp asymmetric Poincaré inequalities in the line. In precise terms, for each $p > 1$, $T > 0$, $a > 0$ and $b \geq 0$, we recall that
\[
\lambda^{a,b}_{p}(-T,T) := \inf_{\substack{u \in W_{0}^{1,p}(-T,T) \\ \lVert u \rVert_{p} = 1}}\ \int_{-T}^T \left(H_{a,b}(u^\prime(t))\right)^p\, dt,
\]
where $H_{a,b}(t) = a t^+ + b t^-$.

The purpose here is twofold: to compute the value of $\lambda^{a,b}_{p}(-T,T)$ and to characterize all corresponding extremizers. Since the nature of the inequalities relies on the value of $b$, we organize this section into two distinct situations: the degenerate case ($b = 0$) and the non-degenerate case ($b > 0$).

\subsection{The case $b = 0$}

It suffices to assume $a = 1$. Observe that $\lambda_{p}^{0,1}(-T,T) = \lambda_{p}^{1,0}(-T,T)$ and also $u$ is an extremizer for $\lambda_{p}^{1,0}(-T,T)$ if and only if $-u$ is an extremizer for $\lambda_{p}^{0,1}(-T,T)$. In particular, our treatment also covers the case $a = 0$.

Set
\[
\lambda^+_{p}(-T,T) = \lambda_{p}^{1,0}(-T,T).
\]
Next we compute the value of $\lambda^+_{p}(-T,T)$ and characterize all corresponding extremizers in a larger space of functions. As a consequence, we will derive the asymmetric Poincaré inequality which is stronger than the classical one,
\begin{equation} \label{SPI}
\int_{-T}^{T} \left( u'(t)^+ \right)^p dt \geq \lambda^+_{p}(-T,T) \int_{-T}^{T} \vert u(t) \vert^p dt
\end{equation}
for every $u \in W_{0}^{1,p}(-T,T)$.

We begin by introducing the piecewise Sobolev space $W_{pw}^{1,p,+}(-T,T)$ as the set of all functions $u: (-T, T) \rightarrow \R$ such that for some finite partition
\[
{\cal P} = \{-T =: T_0 < T_1 < \cdots < T_{m-1} < T_m:= T\},
\]
we have:
\begin{enumerate}[label=(\Alph*)]
\item\label{item1} $u \in W^{1,p}(T_{i-1}, T_i)$ for every $i=1,\ldots,m$,
\item\label{item2} $u(-T) \leq 0$ and $u(T) \geq 0$,
\item\label{item3} $u(T_i^-) > u(T_i^+)$ for every $i=1,\ldots,m-1$ if $m > 1$, where
\[
u(T_i^\pm) := \lim_{t \rightarrow T_i^\pm} u(t).
\]
\end{enumerate}
Clearly, all functions in $W^{1,p}(-T,T)$ that satisfy (B) belong to $W_{pw}^{1,p,+}(-T,T)$.

We begin by ensuring that the sharp inequality \eqref{SPI} can be extended to the larger set $W_{pw}^{1,p,+}(-T,T)$.
\begin{propo} \label{P1}
For any $p > 1$ and $T > 0$, we have
\[
\lambda^+_{p}(-T,T) = \inf_{\substack{u \in W_{pw}^{1,p,+}(-T,T), \\ \lVert u \rVert_{p} = 1}}\ \int_{-T}^T \left(u^\prime(t)^+\right)^p\, dt.
\]
\end{propo}

\begin{proof}
Let $u \in W_{pw}^{1,p,+}(-T,T)$ be a function such that $u(-T) = 0 = u(T)$. For each integer $k \geq 1$ large enough, define the continuous function $u_k(t)$ as
\[
  u_k(t) :=
    \begin{cases}
     u(t) & \text{if $t \in \cup^{m-1}_{i = 1}[T_{i-1},T_i - \frac 1k] \cup [T_{m-1}, T]$},\\
      a_k^i t + b_k^i & \text{if $t \in (T_i - \frac 1k, T_i)$ for}\ i=1, \ldots, m-1,\\
    \end{cases}
\]
where $a_k^i := k\left(u(T_i^+) - u(T_i - \frac 1k)\right)$ and $b_k^i := u(T_i^+) - a_k^i T_i$. Clearly, $u_k \in W^{1,p}_0(-T,T)$ and $a_k^i < 0$ for every $i = 1, \ldots, m-1$ and $k$ large. Since $u \in W^{1,p}(T_{i-1}, T_i)$ for all $i$, we have $u \in L^\infty(-T,T)$ and so letting $k \rightarrow \infty$, we get
\[
\int_{-T}^{T} \vert u_k(t) \vert^p dt \rightarrow \int_{-T}^{T} \vert u(t) \vert^p dt,
\]
\[
\int_{-T}^{T} \left( u_k'(t)^+ \right)^p dt = \sum_{i=1}^{m-1} \int_{T_{i-1}}^{T_i - \frac 1k} \left( u'(t)^+ \right)^p dt + \int_{T_{m-1}}^T \left( u'(t)^+ \right)^p dt \rightarrow \int_{-T}^{T} \left( u'(t)^+ \right)^p dt.
\]
Assume now that $u(-T) < 0$ or $u(T) > 0$. In each case we modify the function $u_k$ in the intervals $[T_0, T_0 + \frac 1k]$ and $[T_m - \frac 1k, T_m]$, respectively. In fact, consider the new function $v_k \in W^{1,p}_0(-T,T)$:
\[
  v_k(t) :=
    \begin{cases}
    a_k^0 t + b_k^0 & \text{if $t \in [T_0,T_0 + \frac 1k]$ in case $u(-T) < 0$},\\
    a_k^m t + b_k^m & \text{if $t \in [T_m - \frac 1k,T_m]$ in case $u(T) > 0$},\\
     u_k(t) & \text{otherwise},\\
    \end{cases}
\]
where
\[
a_k^0 := ku(T_0 + \frac 1k),\ \ b_k^0 := k T u(T_0 + \frac 1k),
\]
\[
a_k^m := -ku(T_m - \frac 1k),\ \ \ b_k^0 := k T u(T_m - \frac 1k).
\]
Since $a_k^0$ and $a_k^m$ are negative for $k$ large, arguing exactly as before, we obtain
\[
\int_{-T}^{T} \vert v_k(t) \vert^p dt \rightarrow \int_{-T}^{T} \vert u(t) \vert^p dt,
\]
\[
\int_{-T}^{T} \left( v_k'(t)^+ \right)^p dt = \int_{T_0 + \frac 1k}^{T_m - \frac 1k} \left( u_k'(t)^+ \right)^p dt \rightarrow \int_{-T}^{T} \left( u'(t)^+ \right)^p dt.
\]
This leads us to the desired statement.
\end{proof}

The main result of this subsection is
\begin{teor} \label{T1}
For any $p > 1$ and $T > 0$, we have:
\begin{itemize}
\item[(i)] $\lambda^+_{p}(-T,T) = \displaystyle \frac{1}{2^p}\lambda_{1,p}(-T,T)$;
\item[(ii)] $u \in W_{pw}^{1,p,+}(-T,T)$ is an extremizer of \eqref{SPI} if and only if either $u = u_p$ or $u = v_p$, where
\[
u_p(t) = c \varphi_p\left(\frac{t-T}{2}\right) \ \ \text{and}\ \ v_p(t) = -c \varphi_p\left(\frac{t+T}{2}\right)\ \ \text{for}\ t \in (-T,T),
\]
where $c$ is any positive constant and $\varphi_p \in W^{1,p}_{0}(-T,T)$ is the principal eigenfunction of the $p$-Laplace operator $- \Delta_p$ associated to the eigenvalue $\lambda_{1,p}(-T,T)$ normalized by $\Vert \varphi_p \Vert_p = 1$.
%characterized variationally by

%\[
%\lambda_{1,p}(-T,T) := \inf_{\substack{u \in W_0^{1,p}(-T,T), \\ \lVert u \rVert_{p} = 1}}\ \int_{-T}^T \vert u^\prime(t) \vert^p\, dt.
%\]
\end{itemize}
\end{teor}
\n As a direct consequence of the above statement, all extremizers of \eqref{SPI} are in $W^{1,p}(-T,T) \setminus W_0^{1,p}(-T,T)$, vanish at one of the endpoints of $(-T,T)$ and have defined sign.

A key ingredient in the proof of the theorem consists in finding a suitable set $X_0 \subset W_{pw}^{1,p,+}(-T,T)$ where possible extremizers must necessarily lie. Consider then the following set
\[
X_0 := \left\{u \in W^{1,p}(-T,T) \colon\  u \ \text{is non-decreasing}, \ \lVert u \rVert_{p} = 1\ \text{and either} \ u(-T) = 0 \ \text{or} \ u(T) = 0\right\}.
\]
\begin{propo} \label{P2} For any $p > 1$ and $T > 0$, we have
\[
\lambda_{p}^{+}(-T,T) = \inf_{\substack{u \in X_0}}\ \int_{-T}^{T} \left( u'(t)^+ \right)^p dt.
\]
\end{propo}

\begin{proof} 
	The proof is based on a three-steps reduction process, transforming the infimum in $W_{pw}^{1,p,+}(-T,T)$ into the infimum in $X_0$.
Observe that any $u \in W_{pw}^{1,p,+}(-T,T)$ must have at least one zero in $[-T,T]$.

\n \textbf{First reduction:} We first restrict $W_{pw}^{1,p,+}(-T,T)$ to functions with a finite number of zeros, since this is a dense subset (this is due to this subset contains the subset of polynomials restricted to $(-T,T)$, which is dense). From now on, we will assume $u^{-1}(0)$ is finite for the functions in $W_{pw}^{1,p,+}(-T,T)$. Next we prove that the infimum can be reduced to the set
\[
X_{2} \coloneqq \left\{u \in W_{pw}^{1,p,+}(-T,T) \colon \ \lVert u \rVert_{p} = 1 \ \text{and} \ u(t) = 0 \ {\rm implies}\ t \in \{-T,T\}\right\}.
\]
Let $u \in W_{pw}^{1,p,+}(-T,T)$ be such that $\lVert u \rVert_{p} = 1$ and assume there is $t_{0} \in (-T,T)$ such that $u(t_{0}) = 0$ (i.e. $u \in X_{2}^{c}$). Clearly, $u \in W_{pw}^{1,p,+}(-T,t_{0}) \cap W_{pw}^{1,p,+}(t_{0},T)$.

Consider the following notations
\[
I_{v} = \int_{-T}^{T} |v|^{p},\ I_{v}' = \int_{-T}^{T} ((v')^{+})^{p}
\]
for any $v \in W_{pw}^{1,p,+}(-T,T)$ and
\[
J_{1} = \int_{-T}^{t_{0}} |u|^{p},\ J_{1}' = \int_{-T}^{t_{0}} ((u')^{+})^{p},\ J_{2} = \int_{t_{0}}^{T} |u|^{p},\ J_{2}' = \int_{t_{0}}^{T} ((u')^{+})^{p} .
\]
Now we introduce the functions $v_1(t)$ and $v_2(t)$ for $t \in (-T,T)$:
\begin{align*}
& v_{1}(t) \coloneqq u\left(\frac{(t_{0}+T)t}{2T}+\frac{t_{0}-T}{2}\right), \\
& v_{2}(t) \coloneqq u\left(\frac{(T-t_{0})t}{2T}+\frac{t_{0}+T}{2}\right).
\end{align*}
Straightforward computations yield
\begin{align*}
& I_{v_{1}} = \frac{2T}{T+t_{0}}J_{1},\ I_{v_{1}}' = \left(\frac{T+t_{0}}{2T}\right)^{p-1}J_{1}', \\
& I_{v_{2}} = \frac{2T}{T-t_{0}}J_{2},\ I_{v_{2}}' = \left(\frac{T-t_{0}}{2T}\right)^{p-1}J_{2}'.
\end{align*}
Assume, without loss of generality, that $\frac{J_{1}'}{J_{1}} \leq \frac{J_{2}'}{J_{2}}$. Otherwise, we exchange $v_1$ by $v_2$ in the argument below. This implies that
\[
\frac{I_{v_{1}}'}{I_{v_{1}}} < \left(\frac{2T}{T+t_{0}}\right)^{p}\frac{I_{v_{1}}'}{I_{v_{1}}} = \frac{J_{1}'}{J_{1}} \leq \frac{J_{1}'+J_{2}'}{J_{1}+J_{2}} = J_{1}'+J_{2}',
\]
and this inequality can be rewritten for $w = \frac{v_{1}}{\lVert v_{1} \rVert_{p}}$ as
\[
\int_{-T}^{T} ((w')^{+})^{p} < \int_{-T}^{T} ((u')^{+})^{p}.
\]
We then repeat the process (now with $w$) finitely many times until the resulting function has no zeros in $(-T,T)$ (this eventually ends because we started with a function with a finite number of zeros). Furthermore, by construction, the resulting function is zero at most at the endpoints $-T$ and $T$ and it has lower energy than the initial one. Therefore, we can reduce the infimum to the set $X_{2}$.

Summarizing, it was shown that
\[
\lambda_{p}^{+}(-T,T) = \inf_{\substack{u \in X_2}}\ \int_{-T}^{T} \left( u'(t)^+ \right)^p dt.
\]

\vspace{0.5cm}

\n \textbf{Second reduction:} Now we prove that the infimum can be reduced to the set
\[
X_1 \coloneqq  \left\{u \in X_{2} \colon\ u \in W^{1,p}(-T,T)\right\}.
\]
Let $u \in X_2$ and assume that $u \not\in W^{1,p}(-T,T)$. Since $u \in W_{pw}^{1,p,+}(-T,T)$, there is a partition
\[
{\cal P} = \{-T = T_{0} < T_{1} < \dots < T_{m-1} < T_{m} = T\}
\]
for some $m > 1$ such that $u$ satisfies \ref{item1}, \ref{item2} and \ref{item3}. Also, since $u \in X_{2}$ ($u$ vanishes at most in $\{-T,T\}$) and due to the conditions \ref{item2} and \ref{item3}, $u$ does not change sign in $(-T,T)$.

Assume that $u(-T) = 0$ and $u > 0$ in $(-T,T)$ (the negative case is similar). Define the jump correction function of $u$ as
\[
j_{u} = \sum_{i=1}^{m-1} \left(u(T_{i}^{-})-u(T_{i}^{+})\right)\chi_{(T_{i},T_{i+1})} \geq 0
\]
and consider the following function in $W^{1,p}(-T,T)$:
\[
v = \frac{u+j_{u}}{\lVert u+j_{u}\rVert_{p}}.
\]
Clearly, we have $v \in X_1$. In addition,
$u+j_u \geq u > 0$ implies $\|u+j_u\| \geq \|u\|_p \geq 1$,
%\begin{align*}
%\lVert u+j_{u}\rVert_{p}^{p} = \int_{-T}^{T} |u(t)+j_{u}(t)|^p \ dt \geq \int_{-T}^{T} |u(t)|^p + |j_{u}(t)|^p \ dt = 1 + \int_{-T}^{T} |j_{u}(t)|^p \ dt \geq 1,
%\end{align*}
so that
\[
\int_{-T}^{T} ((v')^{+})^{p} = \int_{-T}^{T}\left(\frac{(u')^{+}}{\lVert u+j_{u}\rVert_{p}}\right)^{p} \leq \int_{-T}^{T} ((u')^{+})^{p}.
\]
This gives us
\[
\lambda_{p}^{+}(-T,T) = \inf_{\substack{u \in X_1}}\ \int_{-T}^{T} \left( u'(t)^+ \right)^p dt.
\]

\vspace{0.5cm}

\noindent\textbf{Final reduction:} The last step consists in reducing the infimum to the initially introduced set
\[
X_0 = \{u \in X_{1} \colon\ u \ \text{is non-decreasing}\}.
\]
Next we show that a dense subset of $X_{1}$ can be reduced to a dense subset of $X_0$. Take a polynomial $q \in X_{1}$ such that $q(-T) = 0$ and $q > 0$ in $(-T,T)$ (again the case $u < 0$ is similar) and write
\[
\left[(q')^{-1}(-\infty,0)\right] \cap (-T,T) = \bigcup_{i=0}^{N} (t_{2i},t_{2i+1}),
\]
where the intervals in the union are disjoint. Now define the function $u(t)$ for $t \in (-T,T)$ as
\begin{equation*}
u(t) =
\begin{cases}
\begin{aligned}
     & q(t) \ \ \ \ \text{if} \ t \in (q')^{-1}[0,+\infty), \\
     & q(t_{2i}) \ \ \text{if} \ t \in (t_{2i},t_{2i+1}).
\end{aligned}
\end{cases}
\end{equation*}
Note that $u \in W_{pw}^{1,p,+}(-T,T)$ and $u(-T) = 0$. Consider $j_{u}$ as before and define again
\[
v = \frac{u+j_{u}}{\lVert u+j_{u}\rVert_{p}}.
\]
It is clear that $v \in X_0$. Moreover, we have
\begin{align*}
\lVert u+j_{u}\rVert_{p}^{p} 
	&\geq \int_{-T}^{T} |u(t)|^{p} \ dt \\
	&= \int_{(q')^{-1}[0,+\infty)} |q(t)|^{p} \ dt + \sum_{i=0}^{N} \int_{t_{2i}}^{t_{2i+1}} |q(t_{2i})|^p \ dt \\
	&\geq \int_{(q')^{-1}[0,+\infty)} |q(t)|^{p} \ dt + \sum_{i=0}^{N} \int_{t_{2i}}^{t_{2i+1}} |q(t)|^p \ dt \\
	&= \int_{-T}^{T} |q(t)|^p \ dt = 1.
\end{align*}
Therefore,
\[
\int_{-T}^{T} ((v')^{+})^{p} = \frac{1}{\lVert u+j_{u}\rVert_{p}^{p}}\int_{-T}^{T} ((q')^{+})^{p} \leq \int_{-T}^{T} ((q')^{+})^{p}
\]
and so we can reduce the infimum to the set $X_0$. This completes the proof of the proposition.
\end{proof}

Now we are ready to prove the main result of this subsection.
\begin{proof}[Proof of Theorem \ref{T1}] We start assuming by contradiction that
\[
\lambda_{p}^{+}(-T,T) < \frac{\lambda_{1,p}(-T,T)}{2^p}.
\]
From Proposition \ref{P2}, there is a function $u \in X_0$ such that
\[
\int_{-T}^{T} \left( u'(t)^+ \right)^p \ dt < \frac{\lambda_{1,p}(-T,T)}{2^p}.
\]
Assume $u(-T) = 0$ and $u > 0$ in $(-T,T)$ (the negative case is analogous). Define the function $w(t)$ for $t \in (-T,T)$ as
\begin{equation*}
w(t) =
\begin{cases}
\begin{aligned}
     & u(2t+T) \ \ \text{if} \ t \in (-T,0), \\
     & u(T-2t) \ \ \text{if} \ t \in [0,T).
\end{aligned}
\end{cases}
\end{equation*}
It is clear that $w \in W^{1,p}_{0}(-T,T)$. In addition,
\begin{align*}
\int_{-T}^{T} |w(t)|^p \ dt &= \int_{-T}^{0} |u(2t+T)|^p \ dt + \int_{0}^{T} |u(T-2t)|^p \ dt \\
&= \frac{1}{2}\int_{-T}^{T} |u(s)|^p \ ds + \frac{1}{2}\int_{-T}^{T} |u(s)|^p \ ds \\
&= \int_{-T}^{T} |u(s)|^p \ ds \\
&= 1
\end{align*}
and, since $u$ is non-decreasing in $(-T,T)$, we have
\begin{align*}
\int_{-T}^{T} |w'(t)|^p \ dt &= \int_{-T}^{0} 2^{p}|u'(2t+T)|^p \ dt + \int_{0}^{T} 2^{p}|u'(T-2t)|^p \ dt \\
&= 2^{p}\left(\frac{1}{2}\int_{-T}^{T} \left( u'(t)^+ \right)^p \ dt +\frac{1}{2}\int_{-T}^{T} \left( u'(t)^+ \right)^p \ dt \right) \\
	&= 2^{p}\int_{-T}^{T} \left( u'(t)^+ \right)^p \ dt < \lambda_{1,p}(-T,T),
\end{align*}
which contradicts the definition of $\lambda_{1,p}(-T,T)$. Thus,
\[
\lambda_{p}^{+}(-T,T) \geq \frac{\lambda_{1,p}(-T,T)}{2^p}.
\]
Now we define the function $u_{p}(t)$ for $t \in (-T,T)$ as
\[
u_{p}(t) \coloneqq \varphi_{p}\left(\frac{t-T}{2}\right).
\]
Note that $u_{p} \in X_0$. Also, since $\varphi_{p}$ is even in $(-T,T)$ and increasing in $(-T,0)$, we obtain
\begin{align*}
\int_{-T}^{T} |u_{p}(t)|^{p} \ dt
	&= \int_{-T}^{T} \left|\varphi_{p}\left(\frac{t-T}{2}\right)\right|^{p} \ dt \\
	&= 2\int_{-T}^{0} |\varphi_{p}(t)|^{p} \ dt \\
	&= \int_{-T}^{T} |\varphi_{p}(t)|^{p} \ dt \\
	&= 1
\end{align*}
and
\begin{align*}
\int_{-T}^{T} \left( u_p'(t)^+ \right)^p \ dt 
	&= \int_{-T}^{T} \left|\frac{1}{2}\varphi_{p}'\left(\frac{t-T}{2}\right)\right|^{p} \ dt \\
	&= \frac{1}{2^{p}}\left(2\int_{-T}^{0} |\varphi_{p}'(t)|^{p} \ dt\right) \\
	&= \frac{1}{2^{p}} \int_{-T}^{T} |\varphi_{p}'(t)|^{p} \ dt \\
	&= \frac{\lambda_{1,p}(-T,T)}{2^p},
\end{align*}
which implies that
\[
\lambda_{p}^{+}(-T,T) \leq \frac{\lambda_{1,p}(-T,T)}{2^p}
\]
and so the equality occurs. Moreover, the functions $u_p$ and $v_p$ defined in the statement are extremizers of \eqref{SPI}.

We now concentrate on the characterization of all extremizers in the set $W_{pw}^{1,p,+}(-T,T)$. Assume that $u \in W_{pw}^{1,p,+}(-T,T)$ satisfies $\lVert u\rVert_p = 1$ and
\[
\int_{-T}^{T} \left( u'(t)^+ \right)^p dt = \lambda_{p}^{+}(-T,T).
\]
Next we ensure that $u \in X_0$. The proof is carried out in three steps.

\n \textbf{First step:} We claim that the function $u$ does not change sign in $(-T,T)$. Otherwise, by \ref{item1}, \ref{item2} and \ref{item3} in the definition of $W_{pw}^{1,p,+}(-T,T)$, we have $u(t_0) = 0$ for some $t_0 \in (-T,T)$ and so $u \in W_{pw}^{1,p,+}(-T,t_0) \cap W_{pw}^{1,p,+}(t_0,T)$. Consequently,
\begin{align*}
	\lambda_{p}^{+}(-T,T)
	&= \int_{-T}^{T} \left( u'(t)^+ \right)^p dt \\
	&= \int_{-T}^{t_0} \left( u'(t)^+ \right)^p dt + \int_{t_0}^{T} \left( u'(t)^+ \right)^p dt\\
	&= \lambda_{p}^{+}(-T,t_0) \int_{-T}^{t_0} |u(t)|^p dt + \lambda_{p}^{+}(t_0,T) \int_{t_0}^{T} |u(t)|^p dt \\
	&\geq \min\{\lambda_{p}^{+}(-T,t_0),\lambda_{p}^{+}(t_0,T)\},
\end{align*}
which contradicts the strict monotonicity $\lambda_{p}^{+}(-T,T) < \lambda_{p}^{+}(-T,t_0)$ and $\lambda_{p}^{+}(-T,T) < \lambda_{p}^{+}(t_0,T)$ with respect to the domain. The latter follows readily from the positivity of $\lambda_{p}^{+}(-T,T)$, invariance property by translation and the scaling property
\[
\lambda_{p}^{+}(-\alpha T,\alpha T) = \frac{1}{\alpha^p} \lambda_{p}^{+}(-T,T)
\]
for every $\alpha > 0$.

\n \textbf{Second step:} According to the previous step, assume $u(-T) = 0$ and $u > 0$ in $(-T,T)$ (the negative case is analogous). Since $u$ is a positive extremizer in $W_{pw}^{1,p,+}(-T,T)$, the same argument used in the second reduction in the proof of Proposition \ref{P2} gives that $u \in W^{1,p}(-T,T)$, that is, $u$ can not have any jump. Otherwise, since $\lambda_{p}^{+}(-T,T) > 0$, it would be possible to construct a correction function $v \in W_{pw}^{1,p,+}(-T,T) \cap W^{1,p}(-T,T)$ so that $\lVert v\rVert_p = 1$ and
\[
\int_{-T}^{T} \left( v'(t)^+ \right)^p dt < \lambda_{p}^{+}(-T,T).
\]

\n \textbf{Third step:} We now assert that $u$ is increasing in $(-T,T)$. Notice first that $u \in W^{1,p}(-T,T)$ satisfies in the weak sense the Euler-Lagrange equation
\begin{equation} \label{eq1}
- \left( \left( (u')^+ \right)^{p-1} \right)' = \lambda_{p}^{+}(-T,T) u^{p-1}\ \ {\rm in}\ (-T,T),
\end{equation}
that is,
\[
\int_{-T}^T \left( u'(t)^+ \right)^{p-1} \varphi'(t) dt = \lambda_{p}^{+}(-T,T) \int_{-T}^T u(t)^{p-1} \varphi(t) dt
\]
for every $\varphi \in W_0^{1,p}(-T,T)$. Since $u$ is continuous in $(-T,T)$, if $v \in C^1(-T,T)$ is a primitive of $-\lambda_{p}^{+}(-T,T) u^{p-1}$, then the above equality yields
\[
\int_{-T}^T \left( \left( u'(t)^+ \right)^{p-1} - v(t) \right) \varphi'(t) dt = 0
\]
for every $\varphi \in W_0^{1,p}(-T,T)$. Evoking the well-known DuBois-Reymond lemma, taking into account that $\left( (u')^+ \right)^{p-1} \in L^1(-T,T)$, it follows that $\left( (u')^+ \right)^{p-1} = v + c$ in $(-T,T)$ for some constant $c$ and so $\left( (u')^+ \right)^{p-1}$ is continuously differentiable in $(-T,T)$. This implies that the equation \eqref{eq1} is satisfied in the classical sense in $(-T,T)$. Therefore, since the right-hand side of \eqref{eq1} is positive, $(u')^+$ is decreasing in $(-T,T)$. Thus, $u'(t)^+ > 0$ for every $t \in (-T,T)$, once $(u')^+$ is non-negative. Consequently, $u'$ is positive and so $u$ is an increasing $C^1$ function in $(-T,T)$.

Now consider the function
\begin{equation*}
v(t) =
\begin{cases}
\begin{aligned}
     & u(2t+T)\ \ \text{if} \ t \in (-T,0), \\
     & u(T-2t)\ \ \text{if} \ t \in [0,T).
\end{aligned}
\end{cases}
\end{equation*}
Clearly, $v \in W^{1,p}_0(-T,T)$,
\[
\int_{-T}^{T} |v(t)|^p\ dt = 1
\]
and, since $u$ is increasing in $(-T,T)$,
\[
\int_{-T}^{T} |v'(t)|^p \ dt  = 2^{p}\int_{-T}^{T} \left( u'(t)^+ \right)^p dt = \lambda_{1,p}(-T,T).
\]
Therefore, we must have $v = \varphi_{p}$, which implies
\[
u(t) = \varphi_{p}\left(\frac{t-T}{2}\right)\ \ \text{for}\ t \in (-T,T).
\]
This concludes the proof of theorem.
\end{proof}

\subsection{The case $b > 0$}

This subsection is devoted to the complete study of the sharp asymmetric Poincaré inequality
\begin{equation} \label{API}
\int_{-T}^{T} a^p \left( u'(t)^+ \right)^p + b^p \left( u'(t)^- \right)^p dt \geq \lambda^{a,b}_{p}(-T,T) \int_{-T}^{T} \vert u(t) \vert^p dt
\end{equation}
for every $u \in W_0^{1,p}(-T,T)$ and any fixed numbers $a, b > 0$.

Our result is the following
\begin{teor} \label{T2} Let $p > 1$ and $T > 0$. For any $a, b > 0$, we have:
\begin{itemize}
\item[(i)] $\lambda_{p}^{a,b}(-T,T) = \displaystyle \left(\frac{a+b}{2}\right)^{p}\lambda_{1,p}(-T,T)$;
\item[(ii)] $u \in W_0^{1,p}(-T,T)$ is an extremizer of \eqref{API} if and only if either $u = u_p$ or $u = v_p$, where
\begin{equation*}
u_p(t) =
\begin{cases}
\begin{aligned}
     & c \varphi_{p}\left(\frac{T(t-t_{0})}{T+t_{0}}\right)\ \ \text{if} \ t \in (-T,t_{0}), \\
     & c \varphi_{p}\left(\frac{T(t-t_{0})}{T-t_{0}}\right)\ \ \text{if} \ t \in [t_{0},T)
\end{aligned}
\end{cases}
\end{equation*}
and
\begin{equation*}
v_p(t) =
\begin{cases}
\begin{aligned}
     & -c \varphi_{p}\left(\frac{T(t+t_{0})}{T-t_{0}}\right)\ \ \text{if} \ t \in (-T,-t_{0}), \\
     & -c \varphi_{p}\left(\frac{T(t+t_{0})}{T+t_{0}}\right)\ \ \text{if} \ t \in [-t_{0},T),
\end{aligned}
\end{cases}
\end{equation*}
where $c$ is any positive constant, $\varphi_{p} \in W^{1,p}_0(-T,T)$ is the principal eigenfunction of the $p$-Laplace operator normalized by $\Vert \varphi_{p} \Vert_p = 1$ and
\[
t_{0} = \left(\frac{a-b}{a+b}\right)T.
\]
\end{itemize}
\end{teor}

The lemma below will be used in the proof of this theorem:
\begin{lema} \label{L1} Let $(\alpha,\beta) \subset (-T, T)$ be two intervals, let $f : \R \rightarrow \R$ be a continuous function and let $u \in C^{1}[-T,T]$ be a weak solution of
\[
- \Delta_{p}^{a,b} u = f(u) \ \text{in} \ (-T,T).
\]
If $u(\alpha) = 0$, $u'(\beta) = 0$ and $u' \geq 0$ in $(\alpha,\beta)$, then the function
\begin{equation*}
w(t) =
\begin{cases}
\begin{aligned}
     & u(t), \ \text{if} \ t \in [\alpha,\beta), \\
     & u(2\beta - t), \ \text{if} \ t \in [\beta,2\beta - \alpha]
\end{aligned}
\end{cases}
\end{equation*}
belongs to $C^1[\alpha, 2 \beta - \alpha]$ and is a weak solution of
\begin{equation*}
\begin{cases}
\begin{aligned}
& -a^{p}\Delta_{p} w = f(w) \ \text{in} \ (\alpha, 2\beta - \alpha), \\
& w(\alpha) = w(2\beta - \alpha) = 0.
\end{aligned}
\end{cases}
\end{equation*}
\end{lema}
\begin{proof} Clearly, we have $w \in C^1[\alpha, 2 \beta - \alpha]$. Take any $\varphi \in W^{1,p}_{0}(\alpha, 2\beta - \alpha)$ and consider the following functions for every $\varepsilon > 0$:
\begin{equation*}
\phi_{\varepsilon}(t) =
\begin{cases}
\begin{aligned}
     & \varphi(t) \ \ \text{if} \ t \in (\alpha, \beta - \varepsilon), \\
     & -\frac{\varphi(\beta - \varepsilon)}{\varepsilon}(t - \beta) \ \ \text{if} \ t \in [\beta - \varepsilon, \beta), \\
     & 0 \ \ \text{if} \ t \in \mathbb{R} \setminus (\alpha, \beta),
\end{aligned}
\end{cases}
\end{equation*}
\begin{equation*}
\psi_{\varepsilon}(t) =
\begin{cases}
\begin{aligned}
     & \varphi(2\beta - t) \ \ \text{if} \ t \in (\alpha, \beta - \varepsilon), \\
     & -\frac{\varphi(\beta + \varepsilon)}{\varepsilon}(t-\beta) \ \ \text{if} \ t \in [\beta - \varepsilon, \beta), \\
     & 0 \ \ \text{if} \ t \in \mathbb{R} \setminus (\alpha, \beta).
\end{aligned}
\end{cases}
\end{equation*}
It is easy to check that $\phi_{\varepsilon},\psi_{\varepsilon} \in W^{1,p}_0(-L,L)$ and
\begin{align*}
	\lim_{\varepsilon \to 0^{+}} \phi_{\varepsilon}(t) + \psi_{\varepsilon}(2\beta - t) &= \varphi(t) \ \text{for every} \ t \in (\alpha, 2\beta - \alpha),\\
	\lim_{\varepsilon \to 0^{+}} \phi_{\varepsilon}'(t) - \psi_{\varepsilon}'(2\beta - t) &= \varphi'(t) \ \text{for almost every} \ t \in (\alpha, 2\beta - \alpha).
\end{align*}
So, by the dominated convergence theorem, we derive
\begin{equation}\label{limu+1}
\lim_{\varepsilon \to 0^{+}} \int_{\alpha}^{2\beta - \alpha}f(w)\left(\phi_{\varepsilon}(t) + \psi_{\varepsilon}(2\beta - t)\right)\, dt = \int_{\alpha}^{2\beta - \alpha}f(w)\varphi \, dt.
\end{equation}
Also, since $-\Delta_{p}^{a,b}u = -a^{p}\Delta_{p}u$ in $(\alpha,\beta)$, we have
\[
\int_{\alpha}^{\beta} a^{p}|w'|^{p-2} w'\phi_{\varepsilon}' \, dt = \int_{\alpha}^{\beta - \varepsilon} a^{p}|u'|^{p-2}u'\varphi' \, dt + \int_{\beta - \varepsilon}^{\beta} a^{p}|u'|^{p-2}u'\left(-\frac{\varphi(\beta - \varepsilon)}{\varepsilon}\right) \, dt
\]
and since $u'$ is continuous in $[\alpha, \beta]$ and $u'(\beta) = 0$, we also get
\[
\left|\int_{\beta - \varepsilon}^{\beta} a^{p}|u'|^{p-2}u'\left(-\frac{\varphi(\beta - \varepsilon)}{\varepsilon}\right) \, dt\right| \leq a^{p} \lVert \varphi\rVert_{\infty}\lVert u' \rVert_{L^{\infty}[\beta - \varepsilon,\beta]}^{p-1} \xrightarrow[\varepsilon \to 0^{+}]{} 0.
\]
Consequently,
\begin{equation}\label{limu+2}
\lim_{\varepsilon \to 0^{+}} \int_{\alpha}^{\beta} a^{p} |w'|^{p-2} w'\phi_{\varepsilon}' \, dt = \int_{\alpha}^{\beta} a^{p} |w'|^{p-2} w'\varphi' \, dt.
\end{equation}
Now notice that
\begin{align*}
	\int_{\beta}^{2\beta - \alpha} a^{p} &|w'(t)|^{p-2} w'(t) \left(-\psi_{\varepsilon}'(2\beta - t)\right) \, dt \\
	&= a^{p} \int_{\beta}^{2\beta - \alpha} |u'(2\beta - t)|^{p-2}(-u'(2\beta - t))\left(-\psi_{\varepsilon}'(2\beta - t)\right) \, dt \\
	&= a^{p}\int_{\alpha}^{\beta} |u'(t)|^{p-2} u'(t) \psi_{\varepsilon}'(t) \, dt \\
	&= a^{p}\int_{\alpha}^{\beta - \varepsilon} |u'(t)|^{p-2} u'(t)\varphi'(2\beta - t) \, dt + a^{p}\int_{\beta - \varepsilon}^{\beta} |u'(t)|^{p-2} u'(t) \left(-\frac{\varphi(\beta + \varepsilon)}{\varepsilon}\right) \, dt \\
	&=\int_{\beta}^{2\beta - \alpha} a^{p}|w'(t)|^{p-2} w'(t)\varphi'(t) \, dt + a^{p}\int_{\beta - \varepsilon}^{\beta} |u'(t)|^{p-2} u'(t)\left(-\frac{\varphi(\beta + \varepsilon)}{\varepsilon}\right) \, dt
\end{align*}
and, arguing as before, we have
\begin{equation}\label{limu+3}
\lim_{\varepsilon \to 0^{+}} \int_{\beta}^{2\beta - \alpha} a^{p} |w'(t)|^{p-2} w'(t)\left(-\psi_{\varepsilon}'(2\beta - t)\right) \, dt = \int_{\beta}^{2\beta - \alpha} a^{p} |w'|^{p-2} w'\varphi' \, dt.
\end{equation}
Using \eqref{limu+1}, \eqref{limu+2}, \eqref{limu+3} and the fact that $\phi_{\varepsilon}$ and $\psi_{\varepsilon}$ are test functions, we obtain
\begin{align*}
\int_{\alpha}^{2\beta - \alpha} f(w)\varphi \, dt &= \lim_{\varepsilon \to 0^{+}} \int_{\alpha}^{2\beta - \alpha} f(w) \left(\phi_{\varepsilon}(t) + \psi_{\varepsilon}(2\beta - t)\right) \, dt \\
&= \lim_{\varepsilon \to 0^{+}} \int_{\alpha}^{\beta}f(u)\phi_{\varepsilon} \, dt + \int_{\beta}^{2\beta - \alpha} f(u(2\beta - t))\psi_{\varepsilon}(2\beta - t) \, dt \\
&= \lim_{\varepsilon \to 0^{+}}\int_{\alpha}^{\beta} a^{p}|w'|^{p-2} w' \phi_{\varepsilon}' \, dt + \int_{\alpha}^{\beta} f(u) \psi_{\varepsilon} \, dt \\
&= \lim_{\varepsilon \to 0^{+}}\int_{\alpha}^{\beta} a^{p} |w'|^{p-2} w' \phi_{\varepsilon}' \, dt + \int_{\alpha}^{\beta} a^{p}|u'|^{p-2}u'\psi_{\varepsilon}' \, dt \\
&= \lim_{\varepsilon \to 0^{+}}\int_{\alpha}^{\beta} a^{p}|w'|^{p-2} w' \phi_{\varepsilon}' \, dt - \int_{\beta}^{2\beta - \alpha} a^{p} |w'(t)|^{p-2} w'(t) \psi_{\varepsilon}'(2\beta - t) \,dt \\
&= \int_{\alpha}^{\beta} a^{p}|w'|^{p-2} w'\varphi' \, dt + \int_{\beta}^{2\beta - \alpha} a^{p}|w'|^{p-2} w' \varphi' \, dt \\
&= \int_{\alpha}^{2\beta - \alpha} a^{p} |w'|^{p-2} w'\varphi' \, dt.
\end{align*}
\end{proof}
\begin{proof}[Proof of Theorem \ref{T2}] Let $a,b > 0$ be any fixed numbers. By direct minimization theory, there is $u \in W^{1,p}_{0}(-T,T)$ with $\lVert u \rVert_{p} = 1$ such that
\[
\int_{-T}^{T} a^p \left( u'(t)^+ \right)^p + b^p \left( u'(t)^- \right)^p dt = \lambda_{p}^{a,b}(-T,T).
\]
It is easy to see that $u$ is a weak solution of
\begin{equation}\label{pde}
-\left(a^{p}[(u')^{+}]^{p-1}-b^{p}[(u')^{-}]^{p-1}\right)' = \lambda_{p}^{a,b}(-T,T)|u|^{p-2}u\ \ \text{in}\ (-T,T).
\end{equation}
To be more precise, for any $\varphi \in W^{1,p}_{0}(-T,T)$, $u$ satisfies
\[
\int_{-T}^{T}\left(a^{p}[(u'(t))^{+}]^{p-1}-b^{p}[(u'(t))^{-}]^{p-1}\right)\varphi'(t) \ dt =  \lambda_{p}^{a,b}(-T,T)\int_{-T}^{T}|u(t)|^{p-2}u(t)\varphi(t) \ dt.
\]
Standard regularity theory ensures that $u \in C^1[-T,T]$. Notice also that $u$ has definite sign in $(-T,T)$. Otherwise, there is a number $\tilde{t} \in (-T, T)$ such that $u \in W^{1,p}_{0}(-T,\tilde{t})\cap W^{1,p}_{0}(\tilde{t},T)$. Proceeding exactly as we did in \textbf{the first step} of the proof of Theorem~\ref{T1}, we arrive at a contradiction.

Assume from now on that $u \geq 0$ (the case $u \leq 0$ is analogous). By Harnack's inequality (\cite{Tru}), we have $u > 0$ in $(-T,T)$. Since $u(-T) = 0$ and $u(T) = 0$, by Hopf's Lemma, we know that $u'(-T) >  0$ and $u'(T) < 0$. Let $t_0 := \min\{t \in [-T,T]:\, u'(t) = 0\}$ and $t_1 := \max\{t \in [-T,T]:\, u'(t) = 0\}$. Clearly, $-T < t_0 \leq t_1 < T$, $u'(t) > 0$ for every $t \in (-T,t_0)$ and $u'(t) < 0$ for every $t \in (t_1, T)$. We assert that $t_0 = t_1$. Assume by contradiction that $t_0 < t_1$ and consider the following functions
\begin{equation*}
w_1(t) =
\begin{cases}
\begin{aligned}
     & u(t)\ \ \text{if} \ t \in (-T,t_{0}), \\
     & u(2t_{0}-t)\ \ \text{if} \ t \in [t_{0},T+2t_{0}),
\end{aligned}
\end{cases}
\end{equation*}
\begin{equation*}
w_2(t) =
\begin{cases}
\begin{aligned}
     & u(2t_{1}-t)\ \ \text{if} \ t \in (2t_{1}-T,t_{1}), \\
     & u(t)\ \ \text{if} \ t \in [t_{1},T).
\end{aligned}
\end{cases}
\end{equation*}
Note that $w_1, w_2 > 0$ in $(-T,T)$ and that $w_1 \in W^{1,p}_{0}(-T,T+2t_{0})$ and $w_2 \in W^{1,p}_{0}(2t_{1}-T,T)$. Also, since $u$ is a weak solution of \eqref{pde} and $u'(t_0) = u'(t_1) = 0$, by Lemma \ref{L1}, each of them satisfies in the weak sense the respective equations in $(-T,T)$:
\[
-(|w_1'|^{p-2}w_1)' = \frac{\lambda_{p}^{a,b}(-T,T)}{a^{p}}|w_1|^{p-2}w_1
\]
and
\[
-(|w_2'|^{p-2}w_2)' = \frac{\lambda_{p}^{a,b}(-T,T)}{b^{p}}|w_2|^{p-2}w_2,
\]
which implies that $w_1$ is a principal eigenfunction of the $p$-Laplace operator associated to $\lambda_{1,p}(-T,T+2t_{0})$ and $w_2$ is one associated to $\lambda_{1,p}(2t_{1}-T,t_{1})$. Hence,
\[
\lambda_{1,p}(-T,T+2t_{0}) = \frac{\lambda_{p}^{a,b}(-T,T)}{a^{p}}
\]
and
\[
\lambda_{1,p}(2t_{1}-T,t_{1}) = \frac{\lambda_{p}^{a,b}(-T,T)}{b^{p}},
\]
which yields
\[
a^{p}\lambda_{1,p}(-T,T+2t_{0}) = b^{p}\lambda_{1,p}(2t_{1}-T,t_{1}),
\]
or equivalently
\begin{equation}\label{lambdaab1}
\frac{a^{p}}{(2(T+t_{0}))^{p}}\lambda_{1,p}(0,1) = \frac{b^{p}}{(2(T-t_{1}))^{p}}\lambda_{1,p}(0,1).
\end{equation}
Since we are assuming by contradiction that $t_0 < t_1$, the above equality leads us to $t_{1} > \frac{(a-b)}{(a+b)}T$ and thus
\begin{align} \label{lambdaab2}
\lambda_{p}^{a,b}(-T,T) 
	&=  \frac{b^{p}}{(2(T-t_{1}))^{p}}\lambda_{1,p}(0,1) \\
	&> \frac{b^{p}}{\left(2(T- \frac{(a-b)}{(a+b)}T\right))^{p}}\lambda_{1,p}(0,1) \\
	&= \left(\frac{a+b}{4T}\right)^{p}\lambda_{1,p}(0,1) \\
	&= \left(\frac{a+b}{2}\right)^{p}\lambda_{1,p}(-T,T). \nonumber
\end{align}

Now let $\varphi_{p} \in W_0^{1,p}(-T,T)$ be as in the statement and let $t^{*} = \left( \frac{a-b}{a+b}\right)T$. Define the positive function
\[
\tilde{u}(t) =
\begin{cases}
\begin{aligned}
     & \varphi_{p}\left(\frac{T(t-t^{*})}{T+t^{*}}\right)\ \ \text{if} \ t \in (-T,t^{*}), \\
     & \varphi_{p}\left(\frac{T(t-t^{*})}{T-t^{*}}\right)\ \ \text{if} \ t \in [t^{*},T).
\end{aligned}
\end{cases}
\]
It is clear that $\tilde{u} \in W^{1,p}_{0}(-T,T)$ and notice that
\begin{align*}
\int_{-T}^{T} |\tilde{u}|^{p} dt 
	&= \int_{-T}^{t^{*}} \left|\varphi_{p}\left(\frac{T(t-t^{*})}{T+t^{*}}\right)\right|^{p} dt + \int_{t^{*}}^{T} \left|\varphi_{p}\left(\frac{T(t-t^{*})}{T-t^{*}}\right)\right|^{p} dt \\
&= \frac{(T+t^{*})}{T}\int_{-T}^{0}|\varphi_{p}(t)|^{p} dt + \frac{(T-t^{*})}{T}\int_{0}^{T}|\varphi_{p}(t)|^{p} dt \\
	&= \frac{T+t^{*}}{2T}+\frac{T-t^{*}}{2T} \\
	&= 1
\end{align*}
and
\begin{align*}
	\int_{-T}^{T} a^{p} &[(\tilde{u}')^{+}]^{p} + b^{p} [(\tilde{u}')^{-}]^{p} dt \\
&= a^{p}\int_{-T}^{t^{*}}\frac{T^{p}}{(T+t^{*})^{p}}\left|\varphi_{p}'\left(\frac{T(t-t^{*})}{T+t^{*}}\right)\right|^{p} dt + b^{p}\int_{t^{*}}^{T}\frac{T^{p}}{(T-t^{*})^{p}}\left|\varphi_{p}'\left(\frac{T(t-t^{*})}{T-t^{*}}\right)\right|^{p} dt \\
&= a^{p}\frac{T^{p-1}}{(T+t^{*})^{p-1}}\int_{-T}^{0} |\varphi_{p}'(t)|^{p}\ dt + b^{p}\frac{T^{p-1}}{(T-t^{*})^{p-1}}\int_{0}^{T} |\varphi_{p}'(t)|^{p}\ dt \\
&= \frac{a^{p}}{\left(1+\frac{t^{*}}{T}\right)^{p-1}}\frac{\lambda_{1,p}(-T,T)}{2}+\frac{b^{p}}{\left(1-\frac{t^{*}}{T}\right)^{p-1}}\frac{\lambda_{1,p}(-T,T)}{2} \\
%\end{align*}
%
%\begin{align*}
&=\frac{\lambda_{1,p}(-T,T)}{2}\left(\frac{a^{p}}{\left(1+\frac{a-b}{a+b}\right)^{p-1}}+\frac{b^{p}}{\left(1-\frac{a-b}{a+b}\right)^{p-1}}\right) \\
&=\frac{\lambda_{1,p}(-T,T)}{2}\left(\frac{a(a+b)^{p-1}}{2^{p-1}}+\frac{b(a+b)^{p-1}}{2^{p-1}}\right) = \frac{\lambda_{1,p}(-T,T)}{2^{p}}(a+b)(a+b)^{p-1} \\
&= \left(\frac{a+b}{2}\right)^{p}\lambda_{1,p}(-T,T).
\end{align*}
Consequently, $\lambda_{p}^{a,b}(-T,T) \leq \left(\frac{a+b}{2}\right)^{p}\lambda_{1,p}(-T,T)$, which is a contradiction with \eqref{lambdaab2}. Therefore, $t_{0} = t_{1}$ and, from \eqref{lambdaab1}, we have $t_{0} = \left( \frac{a-b}{a+b}\right)T$. Hence,
\begin{align*}
\lambda_{p}^{a,b}(-T,T)
	&= \frac{a^{p}}{\left(2\left(T+\frac{a-b}{a+b}T\right)\right)^{p}}\lambda_{1,p}(0,1)\\
	&= \frac{a^{p}}{\left(\frac{4a}{a+b}T\right)^{p}}\lambda_{1,p}(0,1) \\
	&= \left(\frac{a+b}{2}\right)^{p}\lambda_{1,p}(-T,T).
\end{align*}
Furthermore, using the strict convexity of the function
\[
t \in \R \mapsto a^p (t^+)^p + b^p (t^-)^p,
\]
it is easily checked (e.g. mimicking the proof of Theorem 3.1 in \cite{BFK}) that there is at most one normalized extremizer with the same sign. This implies that $\tilde{u}$ coincides with the desired function $u_{p}$ of the statement.
\end{proof}

\section{Fundamental frequencies for $n = 2$: the non-degenerate case}\label{sec3}

The goal of this section is to prove Theorem \ref{T2.1}.

It suffices to prove that $(a)$ implies $(b)$. Assume by contradiction that $(b)$ is false. Consider the asymmetric convex body ($H \in \mathscr{H}_P$)
\[
D_H = \left\{ (x,y) \in \R^2:\ H(x,y) \leq 1\right\}
\]
and its support function (or dual semi norm)
\begin{equation}
	\label{dualnorm}
H^\circ(x,y) = \max\{xz + yw:\ (z,w) \in D_H\},
\end{equation}
which is also in $\mathscr H_P$. The polar convex body of $D_H$ is defined by
\[
D^\circ_H = \left\{ (x,y) \in \R^2:\ H^\circ(x,y) \leq 1\right\}.
\]

Suppose by contradiction that $H$ is not differentiable at some point $(z,w) \neq (0,0)$.
By homogeneity we know that $H$ is not differentiable at any multiple of $(z,w)$.
Consider the subdifferential of $H$, defined by
\[
	\partial H(z,w) = \left\{(x,y) \in \R^2 \colon\ H(z,w) + x(z-a) + y(w-b) \leq H(a,b) \text{ for every } (a,b) \in \R^2 \right\}.
\]

By the convexity and homogeneity of $H$, $\partial H(z,w)$ is contained in a line not containing the origin, perpendicular to $(z,w)$.
More specifically, by \cite[Theorem 1.7.4]{Schneider} applied to $K = D_H^\circ$, the set $\partial H(z,w)$ can be characterized as the support set of $D_H^\circ$, this is,
\[
	\partial H(z,w) = \left\{(x,y) \in \R^2 \colon\ z x + w y = H(z,w) \ \text{and} \ H^{\circ}(x,y) = 1 \right\}.
\]
Now \cite[Lemma 1.5.14]{Schneider} implies that $\partial H(z,w)$ cannot be a single point, so it must be a non-degenerate segment which we call $S$.
%Clearly the function $H^\circ$ is linear at $C$ (it is $1$ homogeneous and constant in the segment $\partial H(z,w)$).
%By Corollary 1.7.3 applied to $K = D_H$, $\nabla H^\circ(x,y) = (z,w)$ for every $(x,y) \in C$.

Set $\Omega = D_{H}^{\circ}$ and take an eigenfunction $u_p \in W^{1,p}_0(\Omega)$ associated to $\lambda_{p}^{H}(\Omega)$.
Theorem 3.1 of \cite{BCG} in the special case explained in Example 3.4, with $n=2, K = \Omega$ and $A(t) = t^p$, implies that
\begin{equation}\label{ineqconvexsym}
\lambda_{p}^{H}(\Omega) = \iint_{\Omega} H^{p}(\nabla u_p)\, dA \geq \iint_{\Omega} H^{p}(\nabla u_p^{\star})\, dA
\end{equation}
and $\Vert u_p^{\star} \Vert_{p} = \Vert u_p \Vert_{p} = 1$, where $u_p^{\star}$ denotes the convex rearrangement of $u_p$ with respect to $H^\circ$, which also belongs to $W^{1,p}_0(\Omega)$. Then, \eqref{ineqconvexsym} is actually an equality.

Now we shall prove that $\nabla u(x,y)$ is parallel to $(z,w)$ for $(x,y)$ in a set of positive measure, and this shows that the composition $\nabla H( \nabla u(x,y))$ is not well-defined for $(x,y)$ in this set. Notice that we cannot guarantee at this point that $u = u_p^\star$.

By Theorem 3.2 and Example 3.4 of \cite{BCG} (again with $n=2, K = \Omega$ and $A(t) = t^p$), there exists a full-measure set $E \subseteq (\operatorname{essinf} u_p, \operatorname{esssup} u_p)$ such that for every $t \in E$, the level set $\{u \geq t\}$ is homothetic to $\Omega$ up to a set of measure zero, that is, $\{u \geq t\} = a_t \Omega + (x_t, y_t)$ for some positive and decreasing function $a_t$, and vector $(x_t, y_t) \in \R^2$.
Moreover, by \cite[equation (5.15)]{BCG}, for a.e. $t$ and $\mathcal H^1$ a.e. point $(x,y) \in \{u=t\}$, the gradient $\nabla u(x,y)$ is not zero, and $-\frac{\nabla u(x,y)}{|\nabla u(x,y)|}$ is an outer normal vector of $a_t \Omega + (x_t,y_t)$.
By further reducing the set $E$, we may assume that this is the case for all $t \in E$.

Consider the set
\begin{multline*}
	R = \left\{(x,y) \in \R^2: u(x,y) \in E, x z + y w \geq a z + b w \text{ for almost every } (a,b) \text{ such that } u(a,b) \geq u(x,y)\right\}.
\end{multline*}
By the definition of $R$, we know that for $t \in E$,
\begin{align*}
	R \cap \{u=t\}
	&= \{(x,y) \in \R^2: x z + y w \geq a z + b w \text{ for every } (a,b) \in a_t \Omega + (x_t, y_t) \} \\
	&= a_t \{(x,y) \in \R^2: x z + y w \geq a z + b w \text{ for every } (a,b) \in \Omega \} + (x_t, y_t) \\
	&= a_t S + (x_t, y_t).
\end{align*}
We deduce that, for a.e. $(x,y) \in R$, the gradient $\nabla u(x,y)$ is orthogonal to $S$, and thus parallel to $(z,w)$.

By the co-area formula and taking $\varepsilon > 0$ small,
\begin{align*}
	\int_R |\nabla u|
	&= \int_E \mathcal H^1(R \cap \{u = t\}) \\
	&\geq \int_{E \cap [\operatorname{essinf} u_p, \operatorname{esssup} u_p - \varepsilon]} a_t \mathcal H^1(S) \\
	&\geq (\operatorname{esssup} u_p - \operatorname{essinf} u_p - \varepsilon) \, a_{\operatorname{esssup} u_p - \varepsilon} \mathcal H^1(S) \\
	&>0.
\end{align*}
This implies that $R$ has positive measure, and the theorem follows.

\hfill	$\Box$

\section{Fundamental frequencies for $n = 2$: degenerate cases}\label{sec4}

This section is devoted to the proof of Theorems \ref{T3.1} and \ref{T4.1}, among other important results.

Recall that the kernel of an asymmetric anisotropy $H \in \mathscr{H}$ is defined by
\[
\ker(H):= \{(x,y) \in \R^2:\, H(x,y) = 0\}.
\]
Clearly, from the definition of $\mathscr{H}_D$, the closed set $\ker(H)$ is a positive convex cone in $\R^2$, so that the kernel of $H \in \mathscr{H}_D \setminus \{0\}$ can be classified into four category of sets:
\begin{enumerate}[label=(K\arabic*)]
\item\label{PropertyC1} $\ker(H) = \{\alpha e^{i \theta_0}:\, \alpha > 0\}$ for some $\theta_0 \in [0, 2 \pi]$ (a half-line);
\item\label{PropertyC2} $\ker(H) = \{\alpha e^{i \theta_1}:\, \alpha \in \R\}$ for some $\theta_1 \in [0, \pi]$ (a line);
\item\label{PropertyC3} $\ker(H) = \{\alpha e^{i \theta}:\, \alpha > 0,\, \theta \in [\theta_2, \theta_3]\}$ for some $\theta_2, \theta_3 \in [0, 2\pi]$ with $\theta_3 - \theta_2 < \pi$ (a sector);
\item\label{PropertyC4} $\ker(H) = \{\alpha e^{i \theta}:\, \alpha > 0,\, \theta \in [\theta_4, \theta_4 + \pi]\}$ for some $\theta_4 \in [0, \pi]$ (a half-plane).
\end{enumerate}

Some typical examples of functions in $\mathscr{H}_D$ are:
\begin{ex} \label{ex1}
Each example $(Ei)$ below is in correspondence with the above case $(Ki)$:
\begin{enumerate}[label=(E\arabic*)]
\item\label{ExamC1} $H(x,y) = a \left( \vert x \vert^p + \vert y^+ \vert^p \right)^{1/p}$ for a constant $a > 0$ and a parameter $p \geq 1$;
\item\label{ExamC2} $H(x,y) = a y^+ + b y^-$ for constants $a, b > 0$;
\item\label{ExamC3} $\blacktriangleright$ $H(x,y) = a \left( \kappa \vert x \vert + y \right)^+$ for constants $a, \kappa > 0$,\\
$\blacktriangleright$ $H(x,y) = a \left( \vert x^+ \vert^p + \vert y^+ \vert^p \right)^{1/p}$ for a constant $a > 0$ and a parameter $p \geq 1$;
\item\label{ExamC4} $H(x,y) = a y^+$ for a constant $a > 0$.
\end{enumerate}
\end{ex}

A natural question is whether it is possible to characterize analytically in some way all degenerate asymmetric anisotropies. A partial answer is given in the next result where it is exhibited all those corresponding to \ref{PropertyC2} and \ref{PropertyC4}.
\begin{propo} \label{P5}
Let $H \in \mathscr{H}_D \setminus \{0\}$ and set $H_A := H \circ A$ for a matrix $A$. There is a rotation matrix $A$ such that the anisotropy $H_A$ assumes one of the following forms:
\begin{itemize}
\item[(i)] $H_A(x,y) = a y^+ + b y^-$ for certain constants $a, b > 0$ provided that $\ker(H)$ is a line,
\item[(ii)] $H_A(x,y) = a y^+$ for some constant $a > 0$ provided that $\ker(H)$ is a half-plane.
\end{itemize}
\end{propo}

\begin{proof}
Assume that $\ker(H)$ is a line or a half-plane. In each case, let $A$ be a rotation matrix such that $\ker(H_A) = \{(x,0):\, x \in \R\}$ or $\ker(H_A) = \{(x,y):\, x \in \R,\, y \leq 0\}$. Thanks to the sub-additivity of $H$ and the fact that $H(1,0) = 0 = H(-1,0)$, we have
\begin{align*}
	H_A(x,y) &\leq H_A(x,0) + H_A(0,y) = H_A(0,y),\\
	H_A(0,y) &\leq H_A(-x,0) + H_A(x,y) = H_A(x,y),
\end{align*}
so $H_A(x,y) = H_A(0,y)$ and the conclusion follows since $y \mapsto H_A(0,y)$ is an one-dimensional asymmetric anisotropy.
\end{proof}

The following lower estimate always holds for any non-zero asymmetric anisotropy:
\begin{propo} \label{P6}
Let $H \in \mathscr{H} \setminus \{0\}$. There are a rotation matrix $A$ such that $H_A(x,y) \geq \|H\| y^+$ for every $(x,y) \in \R^2$.
\end{propo}

\begin{proof}
Let $(z,w) \in \R^2$ with $\vert (z,w) \vert = 1$ such that $H(z,w) = \|H\|$. We show first that $H^\circ(z,w)$ is finite,  where $H^\circ$ is the dual seminorm defined by \eqref{dualnorm}. Since $(z,w)$ is a point in $\{H = \|H\|\}$ of minimal euclidean norm, the line orthogonal to $(z,w)$ and going through $(z,w)$ must be a support line of the convex set $\{ H \leq \|H\|\}$. This means that $xz + yw \leq 1$ whenever $H(x,y) \leq \|H\|$ and thus $H^\circ(z,w) \leq \|H\|$. By \eqref{dualnorm}, we have
\[
xz + yw \leq H(x,y) H^\circ(z,w)
\]
for every $(x,y), (z,w) \in \R^2$. In particular, if $xz + yw \geq 0$, we have
\[
 \|H\| \left( xz + yw \right)^{+} \leq (H^\circ(z,w))^{-1} \left( xz + yw \right)^{+} = (H^\circ(z,w))^{-1} \left( xz + yw \right) \leq H(x,y).
\]
Also, if $xz + yw < 0$, then $ \|H\|^{-1} \left( xz + yw \right)^{+} = 0 \leq H(x,y)$, finally take a rotation matrix $A$ such that $A^{T}(z,w)=(0,1)$, then
\[
H_{A}(x,y) \geq  \|H\| \left(\langle A(x,y), (z,w) \rangle \right)^{+} = \|H\| \left(\langle (x,y), A^{T}(z,w) \rangle \right)^{+} = \|H\| \left(\langle (x,y), (0,1) \rangle \right)^{+} = \|H\|y^{+}.
\]
\end{proof}

Let $p > 1$ and let $\Omega \subset \R^2$ be a bounded domain. For each rotation matrix $A$ and function $u \in W^{1,p}_{0}(\Omega)$, denote $\Omega_{A} = A^{T}(\Omega)$ and by $u_A$ the function given by $u_{A}(x,y) = u(A(x,y))$, where $A^T$ stands for the transposed matrix of $A$. One easily checks that $H_{A} \in \mathscr{H}$ (resp. $\mathscr{H}_P,\, \mathscr{H}_D$) and $u_A \in W^{1,p}_{0}(\Omega_A)$ for any $H \in \mathscr{H}$ (resp. $\mathscr{H}_P,\, \mathscr{H}_D$) and $u \in W^{1,p}_{0}(\Omega)$.

Some basic properties that follow directly from the definition of $\lambda_{1,p}^{H}(\Omega)$ are:
\begin{enumerate}[label=(P\arabic*)]%[label=(\roman*)]
    \item\label{PropertyP1} If $H \in \mathscr{H}$ and $\alpha \in [0,+\infty)$ then $\alpha H \in \mathscr{H}$ and $\lambda_{p}^{\alpha H}(\Omega) = \alpha^p \lambda_{p}^{H}(\Omega)$ for any bounded domain $\Omega$;

    \item\label{PropertyP2} If $H \in \mathscr{H}$ and $\Omega_{1}$ and $\Omega_{2}$ are bounded domains such that $\Omega_{1} \subseteq \Omega_{2}$ then $\lambda_{p}^{H}(\Omega_{1}) \geq \lambda_{p}^{H}(\Omega_{2})$;

    \item\label{PropertyP3} If $G$ and $H$ are functions in $\mathscr{H}$ such that $G \leq H$, then $\lambda_{p}^{G}(\Omega) \leq \lambda_{p}^{H}(\Omega)$ for any bounded domain $\Omega$;

    \item\label{PropertyP4} If $H \in \mathscr{H}$ then $\lambda_{p}^{H}(\Omega) = \lambda_{p}^{H_A}(\Omega_A)$ for any orthogonal matrix $A$. Moreover, $u \in W^{1,p}_{0}(\Omega)$ is a minimizer of $\lambda_{p}^{H}(\Omega)$ if, and only if, $u_A \in W^{1,p}_{0}(\Omega_A)$ is a minimizer of $\lambda_{p}^{H_A}(\Omega_A)$.
\end{enumerate}

Using these properties along with Theorem \ref{T1} and Proposition \ref{P6}, we can extend the one-dimensional Corollary \ref{C1} to its counterpart in dimension $2$. Precisely, we have:
\begin{propo} \label{P7}
Let $p > 1$, let $\Omega \subset \R^2$ be a bounded domain and let $H \in \mathscr{H}$. The level $\lambda_{p}^{H}(\Omega)$ is positive if, and only if, $H \neq 0$.
\end{propo}

\begin{proof}
Let $H \in \mathscr{H}$ be any asymmetric anisotropy. For $H = 0$, we have $\lambda_{p}^{H}(\Omega) = 0$. For $H \neq 0$, by Proposition \ref{P6}, there exists a rotation matrix $A$ such that $H_A(x,y) \geq \|H\| y^+$ for every $(x,y) \in \R^2$. By Properties \ref{PropertyP1}, \ref{PropertyP3} and \ref{PropertyP4}, we can consider $H(x,y) = H_A(x,y) = y^+$ and $\Omega = \Omega_A$. Furthermore, by Property \ref{PropertyP2}, it suffices to show that $ \lambda_{p}^{H}(Q) > 0$ for any square $Q = (-T,T) \times (-T ,T)$ such that $\Omega \subset Q$.

On the other hand, by the strong Poincaré inequality stated in Theorem \ref{T1}, we obtain
\begin{align*}
\iint_{Q} H^p(\nabla u(x,y))\, dA 
	&= \int_{-T}^T \int_{-T}^T \left( D_y u(x,y)^+ \right)^p\, dy\, dx \\
	&\geq \int_{-T}^T \lambda^+_{p}(-T,T) \int_{-T}^T \vert u(x,y) \vert^p\, dy\, dx  \\
	&=\lambda^+_{p}(-T,T) \iint_{Q} \vert u(x,y) \vert^p\, dA
\end{align*}
for every $u \in C^1_0(Q)$, and so by density we derive $\lambda_{p}^{H}(Q) \geq \lambda^+_{p}(-T,T) > 0$.
\end{proof}

We now concentrate our attention on the proof of Theorems \ref{T3.1} and \ref{T4.1}.

Let $H \in \mathscr{H}_D \setminus \{0\}$ be such that $\ker(H)$ is a line or a half-plane. By Proposition \ref{P5}, in each case there is a rotation matrix $A$ such that $H_{A}(x,y) = a y^+ + b y^-$, where $a, b > 0$ when $\ker(H)$ is a line and $a > 0$ and $b = 0$ when $\ker(H)$ is a half-plane.

The following lemma on dimension reduction in PDEs is required in the proof of the mentioned theorems:
\begin{lema}\label{L2} Let $p > 1$, let $\Omega \subset \R^2$ be a bounded domain and let $H(x,y) = a y^+ + b y^-$ for constants $a > 0$ and $b \geq 0$. Consider the bounded open interval $I = \{x \in \R :\, \Omega_x \neq \emptyset\}$ where $\Omega_{x} = \left\{y \in \R :\, (x,y) \in \Omega\right\}$. If $u \in W^{1,p}_0(\Omega)$ is a weak solution of the equation
\begin{equation} \label{Dpde}
- \Delta_{p}^H u = \lambda^H_{p}(\Omega) \vert u \vert^{p-2}u \ \ \text{in} \ \ \Omega
\end{equation}
then, for almost every $x \in I$, the restriction
\[
u_{x}(y) = u(x,y),\ \forall y \in \Omega_{x},
\]
is a weak solution in $W^{1,p}_0(\Omega_x)$ of the one-dimensional equation
\[
- \Delta_{p}^H v = \lambda^H_{p}(\Omega) \vert v \vert^{p-2} v \ \ \text{in} \ \ \Omega_{x}.
\]
\end{lema}

\begin{proof} Let $u \in W^{1,p}_0(\Omega)$ be a weak solution of \eqref{Dpde}. By Theorem 10.35 of \cite{Le}, we have $u_x \in W^{1,p}_0(\Omega_x)$ for almost every $x \in I$. In particular, $\vert u \vert^{p-2}u \in L^1(\Omega)$ and $\vert u_x \vert^{p-2}u_x \in L^1(\Omega_x)$ for almost every $x \in I$. The proof that the restriction $u_x$ is a weak solution will be organized into two cases.

\n \textbf{First case:} Assume first that $\partial \Omega$ is of $C^1$ class.

For each $x \in I$, the set $\Omega_{x}$ is a finite union of open intervals, so it is enough to prove for the case that $\Omega_{x}$ is an interval to be denoted by $(c_x,d_x)$. Using that $\Omega$ has $C^1$ boundary, the functions
%\begin{align}\label{functionstrip}
%    & x \mapsto c_x \\ \nonumber
%    & x \mapsto d_x \\ \nonumber
%\end{align}
\begin{equation} \label{functionstrip}
x \mapsto c_x\ \ {\rm and}\ \ x \mapsto d_x
\end{equation}
are piecewise $C^1$ in $I$. For each $x \in I$, consider the map
\begin{align*}
\Gamma_x \colon & C^{\infty}_{0}(0,1) \longrightarrow C^{\infty}_{0}\left(\Omega_x\right) \\
& \varphi \longmapsto \varphi\left(\frac{\mathord{\cdot}-c_x}{d_x - c_x}\right).
\end{align*}
Note that $\Gamma_x$ is a bijection for every $x \in I$. Take now any $\varphi \in C^{\infty}_{0}(0,1)$ and choose $x_{0} \in I$ and $\delta_{x_{0}} > 0$, so that the functions in \eqref{functionstrip} are $C^1$ in $(x_{0}-\delta_{x_{0}}, x_{0}+\delta_{x_{0}})$ and in addition
\[
d_x - c_x \geq \frac{d_{x_{0}}-c_{x_{0}}}{2} > 0,\ \ \forall x \in (x_{0}-\delta_{x_{0}}, x_{0}+\delta_{x_{0}}).
\]
Let any $\eta \in C^{\infty}_{0}(x_{0}-\delta_{x_{0}},x_{0}+\delta_{x_{0}})$ and define the function $\phi(x,y) = \eta(x) \Gamma_{x}(\varphi(y))$. Notice that $\phi \in C^{1}_{0}(\Omega)$ and so taking $\phi$ as a test function in \eqref{Dpde}, we have
\[
\int_{x_{0}-\delta_{x_{0}}}^{x_{0}+\delta_{x_{0}}} \int_{\Omega_{x}} G(D_y u(x,y)) \frac{\eta(x) \Gamma_{x}(\varphi^\prime(y))}{d_{x} - c_{x}} dy dx = \lambda^H_{p}(\Omega) \int_{x_{0}-\delta_{x_{0}}}^{x_{0}+\delta_{x_{0}}} \int_{\Omega_{x}} \vert u(x,y) \vert^{p-2} u(x,y) \eta(x) \Gamma_{x}(\varphi(y)) dy dx,
\]
where $G(t) = \left(a^{p}[t^{+}]^{p-1}-b^{p}[t^{-}]^{p-1}\right)$ for $t \in \R$.

On the other hand, this equality can be rewritten as
\[
\int_{x_{0}-\delta_{x_{0}}}^{x_{0}+\delta_{x_{0}}}\left[G(D_y u(x,y)) \frac{\Gamma_{x}(\varphi^\prime(y))}{d_{x} - c_{x}}\, dy - \lambda^H_{p}(\Omega) \int_{\Omega_{x}}\vert u(x,y) \vert^{p-2} u(x,y) \Gamma_{x}(\varphi(y)) \, dy\right] \eta(x)\, dx = 0
\]
for all $\eta \in C^{\infty}_{0}(x_{0}-\delta_{x_{0}},x_{0}+\delta_{x_{0}})$. Since the function
\[
x \mapsto \int_{\Omega_{x}} G(D_y u(x,y)) \frac{\Gamma_{x}(\varphi^\prime(y))}{d_{x} - c_{x}}\, dy - \lambda^H_{p}(\Omega) \int_{\Omega_{x}}\vert u(x,y) \vert^{p-2} u(x,y) \Gamma_{x}(\varphi(y)) \, dy
\]
belongs to $L^{1}_{loc}(x_{0} - \delta_{x_{0}}, x_{0}+\delta_{x_{0}})$, we must have
\[
\int_{\Omega_{x}}G(D_y u(x,y)) \frac{ \Gamma_{x}(\varphi^\prime(y))}{d_{x}-c_{x}}\, dy - \lambda^H_{p}(\Omega) \int_{\Omega_{x}}\vert u(x,y) \vert^{p-2} u(x,y) \Gamma_{x}(\varphi(y)) \, dy = 0
\]
for almost every $x$ in $(x_{0} - \delta_{x_{0}}, x_{0}+\delta_{x_{0}})$.

Consequently, since $\Gamma_x$ is a bijection, the function $u_{x}(y) = u(x,y)$ satisfies
\[
\int_{\Omega_{x}} G(u_{x}^\prime) \psi^\prime\, dy = \lambda^H_{p}(\Omega) \int_{\Omega_{x}} \vert u_{x} \vert^{p-2} u_{x} \psi\, dy
\]
for every $\psi \in C^{\infty}_{0}(\Omega_{x})$ for almost every $x$ in $(x_{0} - \delta_{x_{0}}, x_{0}+\delta_{x_{0}})$. Since we can cover $I$, up to a countable set, by an enumerate union of intervals like $(x_{k} - \delta_{x_{k}}, x_{k} + \delta_{x_{k}})$, the above equality occurs in the general case where $\Omega_{x}$ can be disconnected, and it just says that $u_x$ satisfies in the weak sense
\[
- \Delta_{p}^H u_x = \lambda^H_{p}(\Omega) \vert u_x \vert^{p-2}u_x \ \ \text{in} \ \ \Omega_{x}.
\]

\n \textbf{Second case:} Consider now $\Omega$ an arbitrary bounded domain.

One always can write
\[
\Omega = \bigcup_{j=1}^{\infty} \Omega_{j},
\]
where $\Omega_{j}$ are bounded domains with $C^1$ boundary and $\Omega_{j} \subseteq \Omega_{j+1}$ for all $j \in \N$. If $\psi \in C^{\infty}_{0}(\Omega_{x})$, then there is $j_{0} \in \N$ such that the support of $\psi$ is contained in $(\Omega_{j})_{x}$ for all $j \geq j_{0}$. Thus, from the first part, $u_{x} = u(x,y)$ satisfies
\[
\int_{(\Omega_{j})_{x}} G(u_{x}^\prime) \psi^\prime\, dy = \lambda^H_{p}(\Omega) \int_{(\Omega_{j})_{x}} \vert u_{x} \vert^{p-2} u_{x} \psi\, dy
\]
for almost every $x$ in $I$ (module a countable union of countable sets) and all $j \geq j_{0}$. Taking $j \to \infty$ in the above equality, we derive
\[
\int_{\Omega_{x}} G(u_{x}^\prime) \psi^\prime\, dy = \lambda^H_{p}(\Omega) \int_{\Omega_{x}} \vert u_{x} \vert^{p-2} u_{x} \psi\, dy
\]
almost everywhere for $x \in I$, which proves the lemma.
\end{proof}

\begin{proof}[Proof of Theorems \ref{T3.1} and \ref{T4.1}] Let $\Omega \subset \R^2$ be a bounded domain and let $H \in \mathscr{H}_D$ be an asymmetric anisotropy such that $\ker(H)$ is a line or a half-plane. By Proposition \ref{P5}, there is a rotation matrix $A$ such that $H_A(x,y)  = a y^+ + b y^-$ for some constants $a > 0$ and $b \geq 0$. For convenience, by Property \ref{PropertyP4}, we assume that $H = H_{A}$ and $\Omega = \Omega_{A}$. From the $C^1$ regularity of $H^p$ in $\R \times (\R \setminus \{0\})$, it follows that $\lambda^H_{p}(\Omega)$ is attained in $W^{1,p}_0(\Omega)$ if and only if it is a fundamental frequency.

We first consider the case $b = 0$. Let $\Omega_{x}$ and $I \subset \R$ be as in the previous lemma. Assume by contradiction that $\lambda^H_{p}(\Omega)$ is a fundamental frequency, i.e. \eqref{Dpde} admits a nontrivial weak solution $u \in W^{1,p}_0(\Omega)$. Since $u$ is non-zero, by Lemma \ref{L2}, $u_{x_0} \in W^{1,p}_0(\Omega_{x_0})$, where $u_{x_0}(y) = u(x_0,y)$, is a nontrivial weak solution of
\[
- \Delta_{p}^H u_{x_0} = \lambda^H_{p}(\Omega) \vert u_{x_0} \vert^{p-2}u_{x_0} \ \ \text{in} \ \ \tilde{\Omega}_{x_0}
\]
for some $x_0 \in I$ and some connected component $\tilde{\Omega}_{x_0}$ of $\Omega_{x_0}$. Thus, $\lambda^{a,0}_{p}(\tilde{\Omega}_{x_0}) \leq \lambda^H_{p}(\Omega)$. Since $u_{x_0} \in W^{1,p}_0(\Omega_{x_0})$, by Theorem \ref{T1}, we have $\lambda^{a,0}_{p}(\tilde{\Omega}_{x_0}) < \lambda^H_{p}(\Omega)$. By that same theorem, we can take an interval $[r,s] \subset \tilde{\Omega}_{x_0}$ such that $\lambda^{a,0}_{p}(r,s) < \lambda^H_{p}(\Omega)$. By definition, there is a function $\psi \in W^{1,p}_0(r,s)$ with $\Vert \psi \Vert_p = 1$ such that
\[
\int_{r}^{s} a^p \left( \psi'(y)^+ \right)^p dy < \lambda^H_{p}(\Omega).
\]
Choose now $c < d$ so that $R:= (c,d) \times (r,s) \subset \Omega$ and any positive function $\varphi \in W^{1,p}_0(c,d)$ with $\Vert \varphi \Vert_p = 1$. Then, the function $u(x,y):= \varphi(x) \psi(y)$ for $(x,y) \in R$ belongs to $W^{1,p}_0(R) \subset W^{1,p}_0(\Omega)$ and
\[
\int_\Omega \vert u \vert^p\, dA = \int_R \vert u(x,y) \vert^p\, dA = \left( \int_c^d \varphi(x)^p\, dx \right) \left( \int_r^s \vert \psi(y) \vert^p\, dy \right) = 1,
\]
\[
\int_\Omega H^p(\nabla u)\, dA = \int_R a^p \left( D_y u(x,y)^+\right)^p\, dA = \left( \int_c^d \varphi(x)^p\, dx \right) \left( \int_{r}^{s} a^p \left( \psi'(y)^+ \right)^p\, dy \right) < \lambda^H_{p}(\Omega),
\]
which contradicts the definition of $\lambda^H_{p}(\Omega)$. This proves Theorem \ref{T3.1}.

We now focus on the characterization in the case $b > 0$ given in Theorem \ref{T4.1}.

Assume first that $\lambda^H_{p}(\Omega)$ is a fundamental frequency. Let $I$ be the bounded open interval as above, so (i) in (b) is clearly satisfied. Since $a, b > 0$, Theorem \ref{T2} guarantees the existence of a number $L > 0$ such that $\lambda^{a,b}_{p}(0,L) = \lambda^H_{p}(\Omega)$. For these choices of $I$ and $L$, we show (ii) in (b) by contradiction. Assume that there is a point $x_{0} \in I$ such that $\Omega_{x_0}$ contains a connected component $\tilde{\Omega}_{x_0}$ whose measure is greater than $L$. In this case, we can take a rectangle $R_1 := (c, d) \times (r, s) \subset \Omega$ with $L_1 := s - r > L$. Take also a principal eigenfunction $u_p \in W^{1,p}_0(r,s)$ corresponding to $\lambda^{a,b}_{p}\left( 0,L_1 \right)$ such that $\Vert u_p \Vert_p = 1$ and any positive function $\varphi \in W^{1,p}_0(c, d)$ such that $\Vert \varphi \Vert_p = 1$. Define $u(x,y) := \varphi(x) u_p(y)$ for $(x,y) \in R_1$. It is clear that $u \in W_0^{1,p}(R_1) \subset W_0^{1,p}(\Omega)$ and $\Vert u \Vert_p = 1$. Furthermore,
\begin{align*}
\iint_{\Omega} H^p(\nabla u)\, dA 
	&= \iint_R a^p \left( D_y u(x,y)^+\right)^p + b^p \left( D_y u(x,y)^-\right)^p dy dx \\
	&= \left( \int_{c}^{d} \vert \varphi(x) \vert^p\, dx \right) \left( \int_{r}^{s} a^p \left(u^\prime_p(y)^+\right)^p + \, b^p \left(u^\prime_p(y)^-\right)^p dy \right) \\
	&= \int_{r}^{s} a^p \left(u^\prime_p(y)^+\right)^p + \, b^p \left(u^\prime_p(y)^-\right)^p dy \\
	&= \lambda^{a,b}_{p}(0, L_1) \int_{r}^{s} \vert u_p(y) \vert^p\, dy \\
	&= \lambda^{a,b}_{p}(0, L_1).
\end{align*}
This leads to the contradiction $\lambda^H_{p}(\Omega) \leq \lambda^{a,b}_{p}(0, L_1) < \lambda^{a,b}_{p}(0, L)$.

For the proof of (iii) in (b), we use the assumption that $\lambda^H_{p}(\Omega)$ is a fundamental frequency. Let $u \in W^{1,p}_0(\Omega)$ be a corresponding eigenfunction. By Lemma \ref{L2}, for almost every $x \in I$, $u_{x} \in W^{1,p}_0(\Omega_{x})$ is a weak solution of
\[
- \Delta_{p}^H u_{x} = \lambda^H_{p}(\Omega) \vert u_{x} \vert^{p-2}u_{x} \ \ \text{in} \ \ \Omega_{x}.
\]
In particular, by elliptic regularity, we have $u(x,\cdot) \in C^1(\bar{\Omega}_x)$.

We assert that $u(x,\cdot) = 0$, $u(x,\cdot) > 0$ or $u(x,\cdot) < 0$ in each connected component $\tilde{\Omega}_x$ of $\Omega_x$. Indeed, if $u(x,\cdot)$ is non-zero in some connected component $\tilde{\Omega}_x$, then $\lambda^{a,b}_{p}(\tilde{\Omega}_x) \leq \lambda^H_{p}(\Omega)$. On the other hand, from the previous argument, we already know that strict inequality cannot occur. Therefore, $\lambda^{a,b}_{p}(\tilde{\Omega}_x) = \lambda^H_{p}(\Omega)$ and the claim follows from Theorem \ref{T2}.

We now prove that there is an open subinterval $I^\prime \subset I$ such that either $u(x,\cdot) > 0$ or $u(x,\cdot) < 0$ in some component $\tilde{\Omega}_x$ of $\Omega_x$, and so $\lambda^{a,b}_{p}(\tilde{\Omega}_x) = \lambda^H_{p}(\Omega)$ for every $x \in I^\prime$. Since $u$ is non-zero somewhere in $\Omega$, there is a subset $X \subset I$ with $\vert X \vert > 0$ (i.e. positive Lebesgue measure) such that $u(x,\cdot)$ has sign in some connected component $\tilde{\Omega}_x$ of $\Omega_x$ for every $x \in X$. Similarly to the definition of $I$, let $J$ be the bounded open interval such that $\Omega \subset \R \times J$ and $\Omega^y := \{x \in \R: \, (x,y) \in \Omega\} \neq \emptyset$ for every $y \in J$. We know that $u(\cdot,y) \in W^{1,p}_0(\Omega^y)$ and so, by Morrey's Theorem, $u(\cdot,y) \in C(\bar{\Omega}^y)$ for almost every $y \in J$.

Let $x_0 \in X$ and take $y_0 \in \tilde{\Omega}_{x_0}$ so that $u(\cdot,y_0) \in C(\bar{\Omega}^{y_0})$. Assume $u(x_0,y_0) > 0$ (the negative case is analogous). By continuity, there is a number $\delta > 0$ such that $u(x,y_0) > 0$ for every $x \in I^\prime:= (x_0 - \delta, x_0 + \delta)$. Consequently, $u(x,\cdot) > 0$ in $\tilde{\Omega}_x$ (here $\tilde{\Omega}_x$ is the connected component of $\Omega_x$ such that $(x,y_{0})$ belongs to it), and so $\lambda^{a,b}_{p}(\tilde{\Omega}_x) = \lambda^H_{p}(\Omega) = \lambda^{a,b}_{p}\left(0,L \right)$ for every $x \in I^\prime$, which is equivalent to the equality $\vert \tilde{\Omega}_x \vert = L$  for every $x \in I^\prime$. Furthermore, the $C^{0,1}$ regularity of $\partial \Omega$ implies that $\Omega^\prime:= \{ (x,y) \in \R^2:\, x \in I^\prime,\, y \in \tilde{\Omega}_x \}$ is a $C^{0,1}$ sub-domain of $\Omega$. Hence, the assertion (iii) holds for the number $L$ as defined above.

Conversely, let $\Omega \subset \R^2$ be a bounded domain with $C^{0,1}$ boundary and let $H \in \mathscr{H}_D \setminus \{0\}$ be an anisotropy satisfying the conditions (i), (ii) and (iii) and whose kernel is a line. Let $A$ be a rotation matrix such that $H_A(x,y) = a y^+ + b y^-$ with $a, b > 0$. By Property \ref{PropertyP4}, it suffices to construct a minimizer for $\lambda^{H_A}_{p}(\Omega_A)$ in $W_0^{1,p}(\Omega_A)$.

Again set $\Omega = \Omega_A$ and $H = H_A$. Let $I^\prime$ and $I$ be as in the item (b) of statement and let $\Omega_x$ be as above. Given any $u \in W_0^{1,p}(\Omega)$ with $\Vert u \Vert_p = 1$, since $u(x, \cdot) \in W_0^{1,p}(\Omega_x)$ for almost every $x \in I$, using (ii) and the one-dimensional asymmetric Poincaré inequality, we get
\begin{align*}
\iint_\Omega H^p(\nabla u)\, dA 
	&= \int_{I} \int_{\Omega_x} a^p \left( D_y u(x,y)^+\right)^p + b^p \left( D_y u(x,y)^-\right)^p dy dx \\
	&\geq \int_{I} \lambda^{a,b}_{p}\left(0,L \right) \int_{\Omega_x} \vert u(x,y) \vert^p\, dy dx \\
	&= \lambda^{a,b}_{p}\left(0,L \right) \iint_\Omega \vert u \vert^p\, dA \\
	&= \lambda^{a,b}_{p}\left(0,L \right),
\end{align*}
so that $\lambda^H_{p}(\Omega) \geq \lambda^{a,b}_{p}\left(0,L \right)$.

We now construct a function $u_0 \in W_0^{1,p}(\Omega)$ such that $\Vert u_0 \Vert_p = 1$ and
\[
\iint_\Omega H^p(\nabla u_0)\, dA = \lambda^{a,b}_{p}\left(0,L \right).
\]
Since the boundary of $\Omega$ is of $C^{0,1}$ class and (iii) is satisfied, there are an open interval $I_0 \subset I^\prime$ and a Lipschitz function $g : \bar{I}_0 \rightarrow \R$ such that $\Omega_0:= \{(x,y):\, x \in I_0,\, g(x) < y < g(x) + L\} \subset \Omega^\prime \subset \Omega$. Let $u_p \in W_0^{1,p}(0,L)$ be an eigenfunction corresponding to $\lambda^{a,b}_{p}\left(0,L \right)$ with $\Vert u_p \Vert_p = 1$ and take any positive function $\varphi \in W_0^{1,p}(I_0)$ with $\Vert \varphi \Vert_p = 1$. Define
\begin{equation*}
u_0(x,y) =
\begin{cases}
\begin{aligned}
    &\varphi(x) u_p(y - g(x))\ \ \text{if} \ (x,y) \in \Omega_0, \\
    & 0\ \ \text{if} \ (x,y) \in \Omega\setminus\Omega_{0}.
\end{aligned}
\end{cases}
\end{equation*}
Then, $u_0 \in W_0^{1,p}(\Omega_0) \subset W_0^{1,p}(\Omega)$, since $g$ is Lipschitz, and satisfies
\begin{align*}
\iint_\Omega \vert u_0 \vert^p\, dA 
	&= \iint_{\Omega_0} \vert u_0 \vert^p\, dA \\
	&= \int_{I_0} \int_{g(x)}^{g(x) + L} \varphi(x)^p \vert u_p(y - g(x)) \vert^p\, dy dx \\
	&= \int_{I_0} \int_0^L \varphi(x)^p \vert u_p(y) \vert^p\, dy dx \\
	&= \left( \int_{I_0} \varphi(x)^p\, dx \right) \left( \int_0^L \vert u_p(y) \vert^p\, dy \right) \\
	&= 1,
\end{align*}

\begin{align*}
\iint_\Omega H^p(\nabla u_0)\, dA 
	&= \int_{I_0} \int_{g(x)}^{g(x) + L} \varphi(x)^p \left( a^p \left( u_p^\prime(y - g(x))^+ \right)^p + b^p \left( u_p^\prime(y - g(x))^- \right)^p \right)\, dy dx \\
	&= \int_{I_0} \int_0^L \varphi(x)^p \left( a^p \left( u_p^\prime(y)^+ \right)^p + b^p \left( u_p^\prime(y)^- \right)^p \right)\, dy dx \\
	&= \left( \int_{I_0} \vert \varphi(x) \vert^p\, dx \right) \left( \int_0^L a^p \left( u_p^\prime(y)^+ \right)^p + b^p \left( u_p^\prime(y)^- \right)^p\, dy \right)\\
	&= \lambda^{a,b}_{p}\left(0,L \right) \int_0^L \vert u_p(y) \vert^p\, dy \\
	&= \lambda^{a,b}_{p}\left(0,L \right).
\end{align*}
Therefore, $\lambda^H_{p}(\Omega) = \lambda^{a,b}_{p}\left(0,L \right)$ and $u_0$ is a minimizer to $\lambda^H_{p}(\Omega)$, which is an eigenfunction. This concludes the proof.
\end{proof}

\section{Extended maximization conjecture for $n = 2$}\label{ssec: Maximization conjecture}\label{sec5}

The well-known isoperimetric maximization conjecture for the $p$-Laplace in dimension $n=2$ states that
\[
\sup \{\lambda_{1,p}(\Omega):\, \forall\ \text{bounded domain\ } \Omega \subset \R^2\ {\rm with}\ \vert \Omega \vert = 1\} = \infty.
\]
It is motivated by its validity for $p = 2$ which follows from the knowledge of fundamental frequencies given by
\[
\lambda_{1,2}(\Omega) = \pi^2 \left(\frac{1}{a^2} +  \frac{1}{b^2}\right)
\]
for rectangles under the form $\Omega = (0,a) \times (0,b)$, being computed by Rayleigh in the century XIX.
The above conjecture has recently been proved in dimension $2$ within the context of general norms, see \cite{HM2}.

Next we establish the following extension of the conjecture to any non-zero two-dimensional asymmetric anisotropy and any $p > 1$:
\begin{teor}[isoperimetric maximization] \label{T6}
Let $p > 1$. For any $H \in  \mathscr{H} \setminus\{0\}$, we have
\[
\sup \{\lambda^H_p(\Omega):\, \forall\ \text{bounded domain\ } \Omega \subset \R^2\ {\rm with}\ \vert \Omega \vert = 1\} = \infty.
\]
\end{teor}

The main ingredient in the proof of this result deals with the explicit value of $\lambda^H_p(\Omega)$ for certain degenerate asymmetric anisotropies. Precisely, take $H \in  \mathscr{H} \setminus \{0\}$ so that $\ker(H)$ is a line or a half-plane. As is already known, there is a rotation matrix $A$ such that $H_{A}(x,y) = a y^+ + b y^-$, where $a > 0$ and $b \geq 0$.
\begin{propo} \label{P8}
Let $p > 1$, let $\Omega \subset \R^2$ be any bounded domain and let $H \in \mathscr{H}_D\setminus\{0\}$ and $A$ be as above. Then, we have
\[
\lambda^H_p(\Omega) = \lambda^{a,b}_p(0,L_{\Omega_A}),
\]
where
\[
L_\Omega = \sup \{ \vert J \vert:\,\ \text{for open intervals} \ J \subset \Omega_{x} \ \text{over all} \ x \in I\}.
\]
\end{propo}

The proof of this result makes use of Theorems \ref{T1.1}, \ref{T2.1} and \ref{T4.1} and also the following lemma which will be useful in the next section.
\begin{lema}\label{lemacontinuity} Let $H_0 \in \mathscr{H}_D$ and $H_\varepsilon \in \mathscr{H}$ be anisotropies such that $H_\varepsilon \geq H_0$ for every $\varepsilon > 0$ and $H_\varepsilon \rightarrow H_0$ in $\mathscr{H}$. Then, for any $p > 1$ and any bounded domain $\Omega \subset \R^2$, we have
\[
\lim_{\varepsilon \to 0^{+}} \lambda_{p}^{H_{\varepsilon}}(\Omega) = \lambda_{p}^{H_{0}}(\Omega).
\]
\end{lema}

\begin{proof} First, since $H_{\varepsilon} \geq H_{0}$ for every $\varepsilon > 0$, we have
\[
\liminf_{\varepsilon \to 0^{+}} \lambda_{p}^{H_{\varepsilon}}(\Omega) \geq \lambda_{p}^{H_{0}}(\Omega).
\]
On the other hand, for each $\delta > 0$, let $u_{\delta} \in W^{1,p}_{0}(\Omega)$ with $\lVert u_{\delta} \rVert_{p} = 1$ be such that
\[
\iint_{\Omega} H_{0}^{p}(\nabla u_{\delta})\, dA < \lambda_{p}^{H_{0}}(\Omega) + \delta.
\]
Assuming without loss of generality that
\[
\vert H_\varepsilon(x,y) - H_0(x,y) \vert \leq \varepsilon \vert (x,y) \vert
\]
for every $(x,y) \in \R^2$ and $\varepsilon > 0$, we get
\begin{align*}
\lambda_{p}^{H_{\varepsilon}}(\Omega) - \lambda_{p}^{H_{0}}(\Omega) &\leq \iint_{\Omega}H_{\varepsilon}^{p}(\nabla u_{\delta}) - H_{0}^{p}(\nabla u_{\delta}) \, dA + \delta \\
&\leq p\iint_{\Omega}H_{\varepsilon}^{p-1}(\nabla u_{\delta})\left(H_{\varepsilon}(\nabla u_{\delta})-H_{0}(\nabla u_{\delta})\right) \, dA + \delta \\
&\leq p\iint_{\Omega}H_{\varepsilon}^{p-1}(\nabla u_{\delta})\left(\varepsilon|\nabla u_{\delta}| +H_{0}(\nabla u_{\delta}) - H_{0}(\nabla u_{\delta})\right) \, dA + \delta \\
&= p\varepsilon\iint_{\Omega}H_{\varepsilon}^{p-1}(\nabla u_{\delta})|\nabla u_{\delta}| \, dA + \delta \\
&= p\varepsilon\iint_{\Omega} \left( H_0(\nabla u_{\delta}) + \varepsilon \vert \nabla u_{\delta} \vert \right)^{p-1} |\nabla u_{\delta}| \, dA + \delta
\end{align*}
for $\delta > 0$ fixed and every $\varepsilon > 0$. Hence, letting $\varepsilon \rightarrow 0$, we derive
\[
\limsup_{\varepsilon \to 0^{+}} \lambda_{p}^{H_{\varepsilon}}(\Omega) \leq \lambda_{p}^{H_{0}}(\Omega) + \delta
\]
for every $\delta > 0$. Therefore,
\[
\lambda_{p}^{H_{0}}(\Omega) \leq \liminf_{\varepsilon \to 0^{+}} \lambda_{p}^{H_{\varepsilon}}(\Omega) \leq \limsup_{\varepsilon \to 0^{+}} \lambda_{p}^{H_{\varepsilon}}(\Omega) \leq \lambda_{p}^{H_{0}}(\Omega)
\]
and the conclusion follows.
\end{proof}

\begin{proof}[Proof of Proposition \ref{P8}]
Assume first that $H(x,y) = a y^+ + b y^-$ with $a, b > 0$. By the item (ii) of Theorem \ref{T4.1}, the statement follows directly when $\Omega$ is a rectangle whose sides are parallel to the coordinate axes, that is, $\lambda^H_{p}(\Omega) = \lambda^{a,b}_{p}(0,L_{\Omega})$ for such particular cases of $H$ and $\Omega$ and any $b > 0$.

Now let $\Omega \subset \R^2$ be any bounded domain and let $u \in W_0^{1,p}(\Omega)$. Since $u(x, \cdot) \in W_0^{1,p}(\Omega_x)$ for almost every $x \in I$, we have
\begin{align*}
\iint_\Omega H^p(\nabla u)\, dA 
	&= \int_I \int_{\Omega_x}H^p(\nabla u(x,y))\, dy dx \\
	&= \int_I \int_{\Omega_x} a^p \left( D_y u(x,y)^+\right)^p + b^p \left( D_y u(x,y)^-\right)^p  dy dx \\
	&\geq \int_I \lambda^{a,b}_{p}(0,L_{\Omega}) \int_{\Omega_x} \vert u(x,y) \vert^p  dy dx\\
	&= \lambda^{a,b}_{p}(0,L_{\Omega}) \iint_\Omega \vert u \vert^p\,  dA,
\end{align*}
so that $\lambda^H_{p}(\Omega) \geq \lambda^{a,b}_{p}(0,L_{\Omega})$.

On the other hand, taking a sequence of open rectangles $R_k = (a_k, b_k) \times (c_k, d_k) \subset \Omega$ such that $L_k = d_k - c_k \rightarrow L_\Omega$ as $k \rightarrow \infty$, by Property \ref{PropertyP2} and the previous remark, we obtain
\[
\lambda^H_{p}(\Omega) \leq \lambda^H_{p}(R_k) = \lambda^{a,b}_{p}(0,L_k),
\]
and letting $k \rightarrow \infty$, we derive
\[
\lambda^H_{p}(\Omega) \leq \lambda^{a,b}_{p}(0,L_\Omega).
\]
This concludes the proof for $H(x,y) = a y^+ + b y^-$ with $a, b > 0$ and any bounded domain $\Omega \subset \R^2$. Now letting $b \rightarrow 0$, thanks to Lemma \ref{lemacontinuity} and Theorems \ref{T1.1} and \ref{T2.1}, the same conclusion also occurs for $b = 0$.

Finally, when both $H$ is a degenerate asymmetric anisotropy and $A$ is a matrix as in the statement, by Property \ref{PropertyP4}, we have
\[
\lambda_{p}^{H}(\Omega) = \lambda_{p}^{H_A}(\Omega_A) = \lambda^{a,b}_{p}(0,L_{\Omega_A})
\]
for every $a > 0$, $b \geq 0$ and any bounded domain $\Omega \subset \R^2$.
\end{proof}

\begin{proof}[Proof of Theorem \ref{T6}]
Let any $H \in \mathscr{H} \setminus \{0\}$. By Proposition \ref{P6}, there always exist a rotation matrix $A$ and a constant $a > 0$ such that $H_A(x,y) \geq H_0(x,y) := a y^+$. Consider the open rectangle $\Omega_k = A(R_k)$, where $R_k = (0, k) \times (0, 1/k)$ for each integer $k \geq 1$. Clearly, $\vert \Omega_k \vert = 1$ and, by Properties \ref{PropertyP3} and \ref{PropertyP4} and Proposition \ref{P8}, we get
\[
\lambda_{p}^H(\Omega_k) = \lambda_{p}^{H_A}(R_k) \geq \lambda_{p}^{H_0}(R_k) = \lambda^{a,0}_{p}(0, 1/k).
\]
Letting $k \rightarrow \infty$ in the above inequality, by Theorem \ref{T1.1}, we deduce that $\lambda_{1,p}^H(\Omega_k) \rightarrow \infty$ as wished.
\end{proof}

\section{Anisotropic spectral optimization for $n = 2$}\label{ssec: Spectral optimization}\label{sec6}

This section is dedicated to the proof of Theorems \ref{T5.1} and \ref{T6.1}.

Consider the anisotropic sphere
\[
\s(\mathscr{H}) = \{H \in \mathscr{H}: \Vert H \Vert = 1\}.
\]
We recall that the optimal constants in \eqref{SSE1} are given by
\begin{align*}
	\Lambda^{\max}(\Omega) &= \sup\left\{\lambda_1(-\Delta_p^H):\, -\Delta_p^H \in {\cal E}(\Omega)\ \ {\rm and}\ \ H \in \s(\mathscr{H}) \right\}, \\
	\Lambda^{\min}(\Omega) &= \inf\left\{\lambda_1(-\Delta_p^H):\, -\Delta_p^H \in {\cal E}(\Omega)\ \ {\rm and}\ \ H \in \s(\mathscr{H}) \right\}.
\end{align*}

\begin{proof}[Proof of Theorem \ref{T5.1}]
Let $N(x,y) = \vert (x,y) \vert$. The computation of $\Lambda_{p}^{\max}(\Omega)$ is immediate, since $H \leq N$ in $\R^2$ for every $H \in \s(\mathscr{H})$, $N \in \s(\mathscr{H})$ and $\lambda^N_{p}(\Omega) = \lambda_{1,p}(\Omega)$. In particular, multiples of the $p$-Laplace operator $- \Delta_p$ on $\Omega$ are extremizers for the second inequality in \eqref{SSE1}. Moreover, we claim that these ones are unique provided that all boundary points of $\Omega$ satisfy the Wiener condition. It suffices to show the rigidity of the equality on $\s(\mathscr{H})$.

Let any $H \in \s(\mathscr{H})$ with $H \neq N$. Next we prove that
\[
\lambda^H_{p}(\Omega) < \lambda^{N}_{p}(\Omega).
\]
Consider the positive eigenfunction $\varphi_p \in W^{1,p}_0(\Omega)$ corresponding to $\lambda_{1,p}(\Omega)$ such that $\Vert \varphi_p \Vert_p = 1$. From the interior elliptic regularity theory for the $p$-Laplace operator (\cite{To1}), we have $\varphi_p \in C^1(\Omega)$ and, by assumption, since we are assuming that $\partial \Omega$ satisfies the Wiener condition, the boundary elliptic regularity theory gives $\varphi_p \in C(\overline{\Omega})$ (\cite{KM, LM, Ma}). Using that $\varphi_p = 0$ on $\partial \Omega$, we then show that $B_\varepsilon(0,0) \subset \nabla \varphi_p(\Omega)$ for every $\varepsilon > 0$ small enough. In fact, set $D = {\rm diam}(\Omega)$ and $M = \max_{(x,y) \in \overline{\Omega}} \varphi_p(x,y) > 0$, and take $\varepsilon = D^{-1} M$. Given any $(\eta, \xi) \in B_\varepsilon(0,0)$, consider the function
\[
v(x,y) := \eta x + \xi y - \varphi_p(x,y)\ \ {\rm for}\ \ (x,y) \in \overline{\Omega}.
\]
By continuity and compactness, $v$ admits a global minimum point $(\tilde{x}, \tilde{y}) \in \overline{\Omega}$. We claim that $(\tilde{x}, \tilde{y}) \in \Omega$. Otherwise, let $(\bar{x}, \bar{y}) \in \Omega$ be a global maximum point of $\varphi_p$ in $\overline{\Omega}$. Using the Cauchy-Schwarz inequality, we derive the contradiction
\[
v(\bar{x}, \bar{y}) - v(\tilde{x}, \tilde{y}) = \eta (\bar{x} - \tilde{x}) + \xi (\bar{y} - \tilde{y}) - M \leq \vert (\eta, \xi) \vert\, \vert (\bar{x} - \tilde{x}, \bar{y} - \tilde{y}) \vert - M < \varepsilon D - M = 0.
\]
Therefore, $(\tilde{x}, \tilde{y}) \in \Omega$ and so $\nabla \varphi_p(\tilde{x}, \tilde{y}) = (\eta, \xi)$.

Consequently, for some point $(x_0, y_0) \in \Omega$, we have
\[
H(\nabla \varphi_p(x_0,y_0)) < N(\nabla \varphi_p(x_0,y_0)) = \vert \nabla \varphi_p(x_0,y_0) \vert.
\]
Thus, $H(\nabla \varphi_p) <  \vert \nabla \varphi_p \vert$ in the set $\Omega_\delta := B_\delta(x_0,y_0)$ for some $\delta > 0$ small enough. Using this strict inequality in $\Omega_\delta$ and $H(\nabla \varphi_p) \leq  \vert \nabla \varphi_p \vert$ in $\Omega \setminus \Omega_\delta$, we get
\begin{align*}
\lambda^H_{p}(\Omega) 
	&\leq \iint_\Omega H^p(\nabla \varphi_p)\, dA \\
	&= \iint_{\Omega_\delta} H^p(\nabla \varphi_p)\, dA + \iint_{\Omega \setminus \Omega_\delta} H^p(\nabla \varphi_p)\, dA \\
	&< \iint_\Omega \vert \nabla \varphi_p \vert^p\, dA \\
	&= \lambda_{1,p}(\Omega) \\
	&= \lambda^N_{p}(\Omega).
\end{align*}
\end{proof}

%The notion of optimal design arises of natural manner in the following statement:

%\begin{teor}[lower anisotropic classification] \label{T7}
%Let $\Omega$ be any membrane and let $p > 1$. The optimal anisotropic constant $\lambda_{1,p}^{\min}(\Omega)$ is always positive and given by

%\[
%\lambda_{1,p}^{\min}(\Omega) = \inf_{A} \lambda^+_{1,p}(0, L_{\Omega_A}).
%\]
%Moreover, the infimum is attained if and only if $\Omega$ has optimal anisotropic design. In this case, if the rotation $\tilde{A} = [\tilde{a}_{ij}]$ is a global maximum of $L_{\Omega_A}$, then

%\[
%H(x,y) = (\tilde{a}_{11}\, x + \tilde{a}_{12}\, y)^+
%\]
%is an anisotropic extremizer corresponding to $\lambda_{1,p}^{\min}(\Omega)$. Furthermore, all extremizers have this form provided that $\partial \Omega$ is $C^{0,1}$.
%\end{teor}

The remainder of the section is dedicated to the proof of Theorem \ref{T6.1} by means of the next two propositions. For their statements, it is convenient to introduce the subsets of $\mathscr{H}$:
\begin{align*}
	\mathscr{H}_0 &= \left\{H \in \mathscr{H} \colon \ker(H) \ \text{is a half-plane} \right\}, \\
	\s(\mathscr{H}_0) &= \left\{H \in \mathscr{H}_0 \colon \lVert H \rVert = 1\right\}.
\end{align*}

We also need this result that comes directly from Proposition~\ref{P6}:
\begin{cor}\label{propbound}
For any $H \in \s(\mathscr{H})$, there is $H_0 \in \s(\mathscr{H}_0)$ such that $H \geq H_{0}$.
\end{cor}

\begin{proof} Let $A$ be as in Proposition~\ref{P6} and define $H_{0}(x,y) = (\langle (x,y),A(0,1) \rangle)^{+}$, it is clear that $H_{0} \in \s(\mathscr{H}_0)$. We also have
\[
H(x,y) = H(AA^{T}(x,y)) = H_A(A^{T}(x,y)) \geq \|H\|(\langle A^{T}(x,y), (0,1)\rangle)^{+} = (\langle (x,y),A(0,1) \rangle)^{+} = H_{0}(x,y).
\]
\end{proof}

Next we provide the proof of the last theorem of the introduction.
\begin{proof}[Proof of Theorem \ref{T6.1}] For the equality
\begin{equation} \label{Eq2}
\Lambda^{\min}(\Omega) = \inf_{H \in \s(\mathscr{H}_0)} \lambda_{p}^{H}(\Omega),
\end{equation}
from the definition of $\Lambda_{1,p}^{\min}(\Omega)$ and Corollary \ref{propbound}, note first that
\[
\Lambda^{\min}(\Omega) \geq \inf_{H \in \s(\mathscr{H}_0)} \lambda_{p}^{H}(\Omega).
\]
For the opposite inequality, fix $H \in \s(\mathscr{H}_0)$ and let $A_{H}$ be an orthogonal matrix such that $H_{0}\circ A_{H} = H$. Now set
\[
H_{\varepsilon}(x,y) = \sqrt{\varepsilon(x^2+y^2)+(y^+)^2}
\]
for any $\varepsilon > 0$. Clearly, $(1+\varepsilon)^{-1}H_{\varepsilon} \in {\cal E}(\Omega)\cap\s(\mathscr{H})$ and, from Lemma \ref{lemacontinuity}, we have
\begin{align*}
\lim_{\varepsilon \to 0^{+}} \lambda_{1}\left(-\Delta_{p}^{(1+\varepsilon)^{-1}H_{\varepsilon} \circ A_{H}}\right) 
	&= \lim_{\varepsilon \to 0^{+}} \lambda_{p}^{(1+\varepsilon)^{-1}H_{\varepsilon} \circ A_{H}}(\Omega) \\
	&= \lim_{\varepsilon \to 0^{+}} (1+\varepsilon)^{-p}\lambda_{p}^{H_{\varepsilon} \circ A_{H}}(\Omega) \\
	&= \lim_{\varepsilon \to 0^{+}} (1+\varepsilon)^{-p} \lambda_{p}^{H_{\varepsilon}}(\Omega_{A_{H}^{T}}) \\
	&= \lambda_{p}^{H_{0}}(\Omega_{A_{H}^{T}}) \\
	&= \lambda_{p}^{H}(\Omega),
\end{align*}
so that $\Lambda^{\min}(\Omega) \leq \lambda_{p}^{H}(\Omega)$ for every $H \in \s(\mathscr{H}_0)$ and thus the equality \eqref{Eq2} occurs.

For the statement
\begin{equation} \label{Eq3}
\inf_{H \in \s(\mathscr{H}_0)} \lambda_{p}^{H}(\Omega) = \inf_{\theta \in [0, 2 \pi]} \lambda_{1,p}(0, 2 L_\theta),
\end{equation}
given any $H \in \mathscr{H}_0 \setminus \{0\}$, we have $H(x,y) = c\left( x \cos \theta + y \sin \theta \right)^+$ for some $c > 0$ and $\theta \in [0, 2 \pi]$ and so $H_A(x,y) = c y^+$ for the rotation matrix $A$ of angle $\theta - \frac{\pi}{2}$. By Property \ref{PropertyP1}, assume that $c = 1$. By Property \ref{PropertyP4}, Proposition \ref{P8} and Theorem \ref{T1.1}, we have
\[
\lambda_{p}^{H}(\Omega) = \lambda_{p}^{H_A}(\Omega_A) = \lambda^{1,0}_{p}(0, L_{\Omega_A}) = \lambda_{1,p}(0, 2 L_{\Omega_A}).
\]
On the other hand, since $A(0,1) = (\cos \theta, \sin \theta)$, we can describe $L_{\Omega_A}$ in an alternative way. From the invariance of Lebesgue measure under orthogonal transformations, we have
\begin{align*}
    L_{\Omega_A} &= \sup_{x \in \R}\left\{ |J| \colon J \subset (\{x\}\times\R) \cap \Omega_A \ \text {and} \ J \ \text{is connected} \right\} \\
    &= \sup_{v \in \R^2}\sup \left\{ |J| \colon J \subset \{t(0,1) + v \colon t \in \R\} \cap \Omega_A \ \text{and} \ J  \ \text{is connected}\right\} \\
    &= \sup_{v \in \R^2}\sup \left\{ |J| \colon A J \subset A \{t(0,1) + v \colon t \in \R\} \cap \Omega \ \text{and} \ J  \ \text{is connected}\right\} \\
    &= \sup_{v \in \R^2}\sup \left\{ |A J| \colon A J \subset \{t A(0,1) + Av \colon t \in \R\} \cap \Omega \ \text{and} \ A^{T}J  \ \text{is connected}\right\}\\
    &=\sup_{v \in \R^2}\sup \left\{ |J| \colon J \subset \{t(\cos{\theta},\sin{\theta}) + A^{T}v \colon t \in \R\}\cap \Omega \ \text{and} \ J  \ \text{is connected}\right\} \\
    &= \sup_{v \in \R^2}\sup \left\{ |J| \colon J \subset \{t(\cos{\theta},\sin{\theta}) + v \colon t \in \R\}\cap \Omega \ \text{and} \ J  \ \text{is connected}\right\}  \\
    &= L_\theta.
\end{align*}
Consequently,
\[
\lambda_{p}^{H}(\Omega) = \lambda_{1,p}(0, 2 L_\theta)
\]
for every $H \in \mathscr{H}_0 \setminus \{0\}$ and the equality \eqref{Eq3} follows. In particular, the positivity of $\Lambda^{\min}(\Omega)$ comes from $\inf_{\theta \in [0, 2 \pi]} \lambda_{1,p}(0, 2 L_\theta) > 0$.

Finally, the last infimum is attained for some $\theta_0 \in [0, 2 \pi]$ if and only if the function $\theta \mapsto L_\theta$ has a global maximum, but this just means that $\Omega$ has optimal anisotropic design. In this case, it is clear that all non-zero multiples of $-\Delta_p^{H_0}$, where $H_0(x,y) = (x\cos{\theta_0} + y\sin{\theta_0})^+$, are extremal operators for the first inequality in \eqref{SSE1}.
\end{proof}

\n {\bf Acknowledgments:}
The first author was supported by Grant RYC2021-031572-I, funded by the Ministry of Science and Innovation / State Research Agency / 10.13039 / 501100011033 and by the E.U. Next Generation EU/Recovery, Transformation and Resilience Plan, and by Grant PID2022-136320NB-I00 funded by the Ministry of Science and Innovation.
The second author was partially supported by CAPES (PROEX 88887.712161/2022-00) and the third author was partially supported by CNPq (PQ 307432/2023-8, Universal 404374/2023-9) and Fapemig (Universal APQ 01598-23).

\end{document}